\documentclass{article}
\usepackage{graphicx} 

\title{A practical existence theorem for reduced order models based on convolutional autoencoders}
\author{Nicola Rares Franco$^1$ and Simone Brugiapaglia$^2$}
\date{\footnotesize $^1$MOX, Department of Mathematics, Politecnico di Milano, Italy\\$^2$Department of Mathematics and Statistics, Concordia University, Montreal, QC, Canada}

\usepackage[a4paper, total={4.9in, 7.80in}]{geometry}
\usepackage{amsmath, amsfonts, mathrsfs}
\usepackage{amsthm}
\usepackage{paralist}
\usepackage[misc]{ifsym}
\usepackage{epsfig} 
\usepackage{epstopdf} 
\usepackage[colorlinks=true]{hyperref}
\usepackage[makeroom]{cancel}

\hypersetup{urlcolor=blue, citecolor=red}

\newtheorem{theorem}{Theorem}[section]

\newtheorem{lemma}[theorem]{Lemma}

\newtheorem{definition}[theorem]{Definition}
\newtheorem{remark}[theorem]{Remark}

\DeclareMathOperator*{\argmin}{\arg\!min}
\DeclareMathOperator*{\essinf}{\text{ess}\;\!inf}

\DeclareMathAlphabet{\mathpzc}{OT1}{pzc}{m}{it}
\newcommand{\sterm}{n}
\newcommand{\dnnactivation}{\sigma}
\newcommand{\x}{\mathbf{x}}
\newcommand{\bias}{\mathbf{b}}
\newcommand{\latent}{m}
\newcommand{\hilbert}{\mathscr{V}}
\newcommand{\paramdomain}{\Theta}
\newcommand{\paramcomplex}{\Theta'}
\newcommand{\holo}{\textnormal{Hol}}
\newcommand{\brho}{\boldsymbol{\rho}}
\newcommand{\polyellipse}{\mathcal{E}}
\newcommand{\ha}{\mathcal{HA}}
\newcommand{\operator}{\mathcal{G}}
\newcommand{\mub}{\boldsymbol{\mu}}
\newcommand{\zetab}{\boldsymbol{\zeta}}
\newcommand{\dense}{\mathscr{F}}

\newcommand{\regularizer}{\mathcal{R}}
\newcommand{\expe}{\mathbb{E}}
\newcommand{\probability}{\varrho}
\newcommand{\opnorm}[1]{{\left\vert\kern-0.25ex\left\vert\kern-0.25ex\left\vert #1 
    \right\vert\kern-0.25ex\right\vert\kern-0.25ex\right\vert}}
\newcommand{\weight}{\mathbf{W}}

\newcommand{\review}[1]{#1}

\begin{document}

\maketitle

\renewcommand{\thetheorem}{\arabic{theorem}}

\begin{abstract}
In recent years, deep learning has gained increasing popularity in the fields of Partial Differential Equations (PDEs) and Reduced Order Modeling (ROM), providing domain practitioners with new powerful data-driven techniques such as Physics-Informed Neural Networks (PINNs), Neural Operators, Deep Operator Networks (DeepONets) and Deep-Learning based ROMs (DL-ROMs). 
In this context, 
deep 
autoencoders based on Convolutional Neural Networks (CNNs) have proven extremely effective, 
outperforming established techniques, such as the reduced basis method, when dealing with complex nonlinear problems.
However, despite the empirical success of CNN-based autoencoders, 
there are only a few theoretical results supporting 
these architectures, usually stated
in the form of universal approximation theorems. 
In particular, although the existing literature provides users with guidelines for designing convolutional autoencoders, 
the subsequent challenge of learning the latent features has been barely investigated. Furthermore, 
many 
practical questions 
remain unanswered, e.g., 
the number of snapshots needed for convergence 
or the neural network training strategy. In this work, using recent techniques from sparse high-dimensional function approximation, 
we 
fill some of these gaps by providing 
a new \emph{practical existence theorem} for CNN-based autoencoders when the parameter-to-solution map is holomorphic. This regularity assumption arises 
in many relevant classes of parametric PDEs, such as the parametric diffusion equation, for which we discuss an explicit application of our general theory. 
\end{abstract}


\section{Introduction}

Scientists and engineers 
rely on Partial Differential Equations (PDEs) to model and describe physical phenomena characterizing the behavior of systems, materials, and processes. 
In tandem with efficient numerical solvers, PDE modeling allows engineers to generate robust simulations of physical systems, effectively providing them with reliable tools for forecasting, design, and optimization. 

In practical applications, 
PDE models often involve multiple parameters, which here we 
denote as $\mub\in\mathbb{R}^{p}$, that describe the physical properties of the system and/or specify the scenario under consideration. We can think of, e.g., the viscosity coefficient in a fluid flow simulation \cite{pichi2023artificial}, the morphology of a vascular network in a biophysical model \cite{vitullo2024nonlinear}, or the permeability coefficient in a heat-transfer simulation \cite{brivio2023error}. When these parameter values remain constant, traditional numerical solvers based on, e.g., finite elements, finite differences, or finite volumes, can provide precise and reliable approximations at a computationally feasible expense. However, there are also applications where the model parameters are allowed to change and thus necessitate multiple---fast---simulations. Examples include optimal control (i.e., find the $\mub$ minimizing a given cost functional), inverse problems (i.e., retrieve $\mub$ from sensor measurements), and uncertainty quantification (i.e., $\mub$ is uncertain). For all such \textit{many-query} scenarios, the computational cost entailed by classical solvers becomes prohibitive and constitutes a major limitation.

A popular solution to these issues is 
provided by Reduced Order Models (ROMs). 
They are suitable model surrogates that seek to alleviate the computational burden associated with the aforementioned tasks by constructing low-dimensional representations that are rich enough to capture the essential features of the system. By learning from high-quality samples generated by classical solvers, ROMs can offer precise and efficient predictions, effectively restoring the feasibility of real-time simulations and their applicability to many-query scenarios. However, depending on the problem at hand, constructing accurate and reliable ROMs can be a challenging task. In fact, 
complex model features such as high-dimensional parameter spaces, strong nonlinearities and singular behaviors, pose significant challenges. 
This is well illustrated by all those problems facing the so-called \textit{Kolmogorov barrier} \cite{ahmed2020reduced, barnett2022quadratic, peherstorfer2022breaking}. 
\\\\
Recently, motivated by the impressive success 
of deep learning in a variety of fields including image recognition, natural language processing, and scientific computing, researchers have attempted to 
leverage this methodology to construct effective ROMs, leading to the development of Deep Learning-based ROMs (DL-ROMs)
\cite{franco2023deep, fresca2021comprehensive, pichi2023graph}. Experimentally, these techniques have obtained quite remarkable results. If provided with enough data and properly trained, DL-ROMs can accurately simulate fluid flows \cite{fresca2022pod-dl-rom}, as well as complex biological \cite{franco2023mesh-informed, vitullo2024nonlinear} and mechanical phenomena \cite{rosafalco2021online}. Furthermore, if complemented with suitable \emph{ad hoc} strategies, they can incorporate some key physical properties of the underlying system such as local mass conservation \cite{boon2023deep}.

In general, this research line is part of a broader trend concerning the development of deep learning algorithms for operator learning in high-dimensional spaces: in fact, a surrogate model can be interpreted as an approximation of the parameter-to-solution map of a given PDE. In this sense, there is 
a close connection between DL-ROMs and 
techniques such as DeepONets \cite{lu2021learning} and (Fourier) Neural Operators \cite{kovachki2021neural, li2020fourier}. For instance, the DeepONet algorithm can be regarded as a space-continuous version of the so-called POD-NN ROM \cite{hesthaven2018non-intrusive}, a predecessor of the POD-DL-ROM \cite{fresca2022pod-dl-rom}. Similarly, the architectures implemented in some DL-ROMs can be traced back to discrete equivalents of certain Neural Operators \cite{franco2023latent, franco2023mesh-informed}. Here, however, we shall limit our attention to the case of DL-ROMs, 
adopting a perspective commonly 
accepted in the ROM literature.
\\\\
At first, the success 
of deep learning in reduced order modeling was mostly empirical (see, e.g., \cite{fresca2021comprehensive, hesthaven2018non-intrusive, lee2020model}). However, with new mathematical insights on neural network approximation theory, such as the seminal paper 
\cite{yarotsky2017error} and subsequent developments (see, e.g., \cite{elbrachter2021deep} and references therein), DL-ROMs are now starting to develop theoretical foundations. \review{Relevant contributions in this direction include 
\cite{kutyniok2022theoretical,marcati2023exponential,schwab2023deep}, 
which are theoretical works characterized by a
major focus on approximation theory  
(practical details concerning networks type or training strategies are not addressed),} 
research on DeepONets \cite{lanthaler2022error} and Neural Operators \cite{kovachki2021universal}, and recent results on autoencoder-based ROMs (see, e.g., \cite{brivio2023error, franco2023latent, franco2023deep, liu2024generalization}).

Here, we shall focus 
on the latter class of DL-ROMs, i.e., deep learning-based surrogate models that reduce the problem complexity by leveraging deep convolutional autoencoders \cite{cheng2018deep}.
This choice is motivated by the fact that, despite being extremely popular among researchers, ROMs based on convolutional autoencoders are still lacking a comprehensive theoretical foundation. While issues like the role of convolutional blocks \cite{franco2023approximation} or the choice of the latent dimension \cite{franco2023latent, franco2023deep}  are relatively well-understood, some key practical questions are still open, especially when it comes to the actual training of these architectures. 
Our purpose for this work is to take a step further and extend the existing literature by offering additional insights on DL-ROM training,  with a particular focus on 
the challenge of learning convolutional features. To this end, we 
propose a new analysis of these architectures based on the framework of \textit{practical existence theorems}. This new paradigm was recently introduced 
in \cite{adcock2022deep, adcock2021gap} for scalar- and Hilbert-valued approximation and further extended to Banach-valued functions in \cite{adcock2022learning} \review{(see also the recent review paper \cite{adcock2024learning})}. It leverages recent advances in sparse high-dimensional polynomial approximation theory \cite{adcock2022sparse} and complements existence results for neural networks (commonly referred to as \emph{universal approximation theorems}) with more practical insights on model training, regularization and sampling.





\subsection{Main contributions} 
Let $\Omega\subset\mathbb{R}^{d}$ be a bounded domain, and let 
\[\paramdomain\ni\mub\mapsto u_{\mub}\in H^{s}(\Omega)\]
be the parameter-to-solution map of a parametrized PDE, where $\paramdomain\subset\mathbb{R}^{p}$ is the parameter space and $H^{s}(\Omega)$ denotes a suitable Sobolev space with smoothness index $s\in\mathbb{N}.$ A classical numerical solver based on, e.g., finite elements or finite differences, provides access to pointwise approximations of the PDE solution over a collection of nodes $\x_{1},\dots,\x_{N_{h}}\in\overline{\Omega},$ with $N_{h}$ being the total number of vertices constituting the spatial grid. 

Therefore, we can think of the numerical solver, also referred to as Full Order Model (FOM) in the reduced order modeling literature, as a map
\[\paramdomain\ni\mub\mapsto[u_{\mub}(\x_{1}),\dots,u_{\mub}(\x_{N_{h}})]^\top\in\mathbb{R}^{N_{h}}.\]
The purpose of DL-ROMs is to construct a Deep Neural Network (DNN) model $\Phi:\mathbb{R}^{p}\to\mathbb{R}^{N_{h}}$ such that $\Phi_{j}(\mub)\approx u_{\mub}(\x_{j})$, where $\Phi_{j}$ denotes the $j$th output neuron of $\Phi$. In the case of autoencoder-based approaches, the construction of $\Phi$ 
relies on three neural network models, namely,
\[\Psi':\mathbb{R}^{N_{h}}\to\mathbb{R}^{\latent},\quad\quad\Psi:\mathbb{R}^{\latent}\to\mathbb{R}^{N_{h}},\]
\[\phi:\mathbb{R}^{p}\to\mathbb{R}^{\latent}.\]
The first two models, the encoder and the decoder, respectively, are trained such that
\[\Psi(\Psi'([u_{\mub}(\x_{1}),\dots,u_{\mub}(\x_{N_{h}})]^\top))\approx [u_{\mub}(\x_{1}),\dots,u_{\mub}(\x_{N_{h}})]^\top.\]
In this way, by leveraging the autoencoder $\Psi\circ\Psi'$, the spatial features characterizing the solutions to the PDE can be synthesized using a smaller number of degrees of freedom, known as the ``latent'' variables. In fact, each discrete vector $\mathbf{u}_{\mub}:=[u_{\mub}(\x_{1}),\dots,u_{\mub}(\x_{N_{h}})]^\top$ can now be represented as $\Psi'(\mathbf{u}_{\mub})\in\mathbb{R}^{\latent}$, with a substantial reduction in complexity whenever $m\ll N_{h}.$ 

Conversely, the third network, $\phi$, sometimes also referred to as \textit{reduced network}, is trained to learn the parameter-to-latent-variables map,
\[\phi(\mub)\approx \Psi'([u_{\mub}(\x_{1}),\dots,u_{\mub}(\x_{N_{h}})]^\top).\]
Once all DNN modules have been trained, the encoder block $\Psi'$ can be discarded and the DL-ROM constructed by composition
, i.e.,
\[\Phi:=\Psi\circ\phi.\]
Given a new parametric instance $\mub\in\paramdomain$, the reduced network computes the corresponding latent solution, namely $\phi(\mub)$, which is then expanded by the decoder to retrieve the final output.
In other words, methods based on autoencoders are grounded on the idea of splitting the complexity of the problem into two components. On the one hand, we have the spatial complexity of PDE solutions, tackled by $\Psi'$ and $\Psi$. On the other hand, there is the inherent complexity associated with the parameter dependence of PDE solutions, addressed by $\phi$.

The existing literature provides insights on the choice of the latent dimension $\latent$ (see, e.g., \cite{franco2023deep}) and on the type of architectures, favoring the use of Convolutional Neural Networks (CNNs)---whenever possible---for the decoder module (see, e.g., \cite{franco2023approximation}). It is worth mentioning that, while the works \cite{franco2023approximation, franco2023deep} are purely theoretical, their conclusions 
are perfectly aligned with empirical evidence. In fact, researchers had long conjectured that autoencoders could compress solutions to their intrinsic dimension, dictated by the number of parameters \cite{fresca2021comprehensive, lee2020model}. Similarly, by leveraging the heuristic observation that 
discrete signals defined over hypercubic domains are roughly equivalent to RGB images, several authors had suggested the use of CNNs for the autoencoder module \cite{fresca2021comprehensive, mucke2021reduced}. Nonetheless, little is know 
about the training of these architectures, in terms of, e.g., sample size and choice of the loss function. Our main contribution, which is fully detailed in Theorem~\ref{theorem:cnn}, goes precisely in this direction.

Given a probability distribution $\probability$ over the parameter space, for each of the three architectures, $\phi$, $\Psi$ and $\Psi'$, we identify a specific class of neural network models, with the decoder $\Psi$ being convolutional, such that, with high probability, the trained DL-ROM satisfies an error bound of the form
\begin{equation*}
\expe_{\mub\sim\probability}^{1/2}\left[\;\sup_{j=1,\dots,N_{h}}|u_{\mub}(\x_{j})-\Psi_{j}(\phi(\mub))|^{2}\right]\le C\left(\sqrt{m}e^{-\frac{1}{\sqrt{2}}\gamma\tilde{N}^{1/(2p)}}+\sqrt{\frac{2m^{1-2s}}{2s-1}}\right),
\end{equation*}
where $m$ is (proportional to) the latent dimension, $\tilde{N}$ is the sample size (up to log factors), \review{whereas $\gamma>0$  and $C>0$ are two constants related to the regularity and the magnitude of the solution operator, respectively; finally, we recall, that $p$ and $s$ are the number of parameters and the smoothness of the PDE solutions, respectively.} In doing so, we also specify the sampling and the optimization procedure associated with the reduced network $\phi$, identifying a specific loss function and a corresponding regularization criterion, thus making our existence result \textit{practical}. Note that, although based on the same ideas presented in \cite{adcock2022deep}, our result is somewhat stronger. In fact, with respect to the space variable $\x$, it provides a uniform error bound, as opposed to a space-averaged one.

At its core, our derivation leverages 
the theory of 
sparse polynomial approximation of high-dimensional, holomorphic maps, and thus relies on the assumption that the parameter-to-solution map admits a suitable holomorphic extension. However, as we will explore later, this assumption is not overly restrictive. In fact, there are numerous practical cases where this condition holds true, such as the parametric diffusion equation, for which an application of Theorem~\ref{theorem:cnn} is explicitly discussed in Section~\ref{sec:param_diff}. 
Finally, we mention that our analysis is limited to the one-dimensional case, $d=1$. Generalizations to higher-dimensional domains are in order but out of the scope of this work.


\subsection{Outline}
The paper is organized as follows. First, in Section~\ref{sec:preliminaries}, we set the notation and introduce some of the basic mathematical concepts upon which our analysis in constructed, such as holomorphic extensions and neural network models. Then, in Section~\ref{sec:main} we present our main result, Theorem~\ref{theorem:cnn}, and its application to the parametric diffusion equation. The proof of the theorem, which is comprised of multiple steps, is postponed to Section~\ref{sec:proof}, together with some auxiliary results that are necessary for our construction (only some of them: the most technical ones are deferred to  Appendix~\ref{appendix:hermite} and \ref{appendix:relu}). Lastly, Section~\ref{sec:conclusions} is dedicated to a final discussion of our findings and potential avenues for future research.

\section{Preliminaries and notation}
\label{sec:preliminaries}
In this section we introduce the main notions and definitions needed to carry out our analysis. In Section~\ref{sec:holo}, we introduce the concepts of holomorphic extension 
and of hidden anisotropy. Section~\ref{sec:nn}, instead, provides the essential background on feedforward and convolutional neural networks.

\subsection{Holomorphic regularity assumption}
\label{sec:holo}

One of the key ingredients of our study is the notion of holomorphic extension. Let $\paramcomplex\subseteq\mathbb{C}^{p}$ be an open set and let $\hilbert$ be a Hilbert space. We denote by $\holo(\paramcomplex,\hilbert)$ the set of holomorphic maps from $\paramcomplex$ to $\hilbert$. More precisely, $f\in\holo(\paramcomplex,\hilbert)$ if and only if the following limit exists for all $\mathbf{z}\in\paramcomplex$ and all directions $j=1,\dots,p$:
\[\lim_{\substack{h\in\mathbb{C}\\h\to0}}\frac{f(\mathbf{z}+h\mathbf{e}_{j})-f(\mathbf{z})}{h}\in\hilbert,\]
where $\mathbf{e}_{j}=(\delta_{i,j})_{i=1}^{p}$ and $\delta_{i,j}$ denotes the Kronecker delta.

\begin{definition} {\bf(Holomorphic extension)}
    Let $(\hilbert, \|\cdot\|)$ be a Hilbert space and let $\paramdomain\subseteq\mathbb{R}^{p}$ be a set. Let $K\subseteq\mathbb{C}^{p}$ be a closed set such that $\paramdomain\subseteq K$. We say that a map $f:\paramdomain\to\hilbert$ admits a \textit{holomorphic extension} to $K$ if there exists an open set $\paramcomplex$, $K\subseteq\paramcomplex\subseteq\mathbb{C}^{p}$, and a holomorphic map $\tilde{f}\in\holo(\paramcomplex,\hilbert)$ such that $\tilde{f}_{|\paramdomain}=f.$ In this case we also set
    \begin{align}
        \|f\|_{L^{\infty}(K,\hilbert)}:=\inf\Big\{\|\tilde{f}\|_{L^{\infty}(\paramcomplex,\hilbert)}\;\;\text{s.t.}\;\;&\paramcomplex\;\text{open},\;K\subseteq\paramcomplex\subseteq\mathbb{C}^{p},\\    \nonumber&\tilde{f}\in\holo(\paramcomplex,\hilbert),\;\tilde{f}_{|\paramdomain}=f\Big\}.
    \end{align}
\end{definition}
Specifically, we are interested in maps that admit holomorphic extensions to so-called  Bernstein polyellipses. 
\begin{definition} {\bf(Bernstein polyellipse)}
    Let $\brho=(\rho_{i})_{i=1}^{p}\in(1,+\infty)^{p}$. We call the set
    \[\polyellipse_{\brho}:=\polyellipse_{\rho_{1}}\times\dots\times\polyellipse_{\rho_{p}}\subset\mathbb{C}^{p}\]
    \review{a} \emph{Bernstein polyellipse} of parameter $\brho$, where $\polyellipse_{\rho}:=\{\review{\frac{z+z^{-1}}{2}}:z\in\mathbb{C},\;1\le|z|\le\rho\}.$
\end{definition}
This setting is justified by the fact that several families of parametric models based on differential equations have parameter-to-solution maps admitting holomorphic extensions to Bernstein polyellipses. These include parametric diffusion problems, parametric parabolic problems, PDEs over parametrized domains, and parametric initial-value problems. For further discussion, we refer to, e.g., \cite[Chapter 4]{adcock2022sparse} and \cite{cohen2015approximation}. 
\\\\
If a map $f$ admits a holomorphic extension to a polyellipse $\mathcal{E}_{\brho}$, the parameter $\brho$ acts as a measure of its \textit{anisotropy}, i.e., it gauges the smoothness of $f$ with respect to each input variable, and it may or may not be known \emph{a priori}. Here, we focus on the more realistic case of unknown or \emph{hidden} anisotropy (cf.\ \cite[Definition 4]{adcock2022deep}).

\begin{definition} {\bf(Hidden anisotropy)}
    \label{def:hidden}
    Let $\hilbert$ be a Hilbert space, $\paramdomain = [-1,1]^p \subset\mathbb{R}^{p}$, $\gamma>0$ and $\epsilon>0$. 
    We write $\ha_{\gamma,\epsilon}(\paramdomain;\hilbert)$ for the set of Hilbert-valued maps $f:\paramdomain\to \hilbert$ which admit a holomorphic extension to a Bernstein polyellipse $\polyellipse_{\brho}\supset\paramdomain$ 
    whose parameter $\brho=(\rho_{j})_{j=1}^{d}$ satisfies   
    \begin{equation}   
    \label{eq:condition_rho}
    p!\prod_{j=1}^{p}\log(\rho_{j})\ge\gamma^{p}(p+1)^{p}(1+\epsilon)^{-1}.
    \end{equation}
\end{definition}
Although the definition of hidden anisotropy might seem obscure or somewhat arbitrary, there is a clear rationale for it. If a map $f:\paramdomain = [-1,1]^p\to \hilbert$ admits a holomorphic extension to a Bernstein polyellipse $\mathcal{E}_{\brho}$, then its \emph{best $\sterm$-term approximation} $f_\sterm$ with respect to Legendre orthogonal polynomials on $L^2(\paramdomain;\hilbert)$ satisfies the following exponential decay rate for any $\epsilon > 0$ (see, e.g., \cite[Theorem 3.15]{adcock2022sparse}):
\[
\|f-f_\sterm\|_{L^2(\paramdomain;\hilbert)}
\leq \exp\left(- C_{\epsilon, p, \brho} \cdot \sterm ^{1/p}\right),
\]
for $\sterm$ large enough (more precisely, for $\sterm \geq \bar{\sterm}$ where $\bar{\sterm}=\bar{\sterm}(\epsilon, p, \brho)$) and where \[
C_{\epsilon, p, \brho} = \frac{1}{p+1}\left(\frac{p! \prod_{j=1}^p \log(\rho_j)}{1+\epsilon}\right)^{1/p}.
\]
Condition \eqref{eq:condition_rho} of Definition~\ref{def:hidden} simply ensures a uniform control of the constant $C_{\epsilon, p, \brho}$ via the inequality \review{$C_{\epsilon, p, \brho} \geq \gamma$}. In other words, all functions $f \in \ha_{\gamma,\epsilon}(\paramdomain;\hilbert)$ satisfy the same best $\sterm$-term exponential decay rate $\|f-f_\sterm\|_{L^2(\paramdomain;\hilbert)}
\leq \exp\left(- \gamma \cdot \sterm ^{1/p}\right)$,  for $\sterm$ large enough. For further details we refer to \cite[Section~3]{adcock2022deep} and references therein.

Since our main focus will be on solution operators to parametrized PDEs, we find it convenient to introduce a short-cut notation for maps taking values in Sobolev spaces. We report it below.

\begin{definition} {\bf(Hidden anisotropy for Sobolev-valued maps)}
    \label{def:hidden-sobolev}
    For $\Omega=(0,1)$, $\paramdomain = [-1,1]^p \subset\mathbb{R}^{p}$, $\gamma>0$, $\epsilon>0$ and $s\in\mathbb{N}$, we set 
    \[\ha_{\gamma,\epsilon,s}(\paramdomain)=\ha_{\gamma,\epsilon}(\paramdomain;H^{s}(\Omega)).\] 
\end{definition}

In practice, Definition~\ref{def:hidden-sobolev} simply sets $\hilbert:=H^{s}(\Omega)$. Here, following usual conventions, we equip $H^{s}(\Omega)$ with the energy norm
\[\|u\|_{H^{s}(\Omega)}:=\sqrt{\int_{0}^{1}|u|^{2}dx+\sum_{k=1}^{s}\int_{0}^{1}\left|\frac{d^{k}u}{dx^{k}}\right|^{2}dx}.\]
We point out that, while Definition~\ref{def:hidden-sobolev} allows for $s=0$, our attention will be devoted to smoother scenarios, namely $s\ge1$.

\subsection{Background on neural networks}
\label{sec:nn}

We now 
recall the mathematical definition of some of the most classical neural network architectures. We start with (feedforward) Deep Neural Networks (DNNs) implementing ``standard'' layers. Hereon, we adopt the usual convention according to which scalar functions are allowed to operate on vectors by acting componentwise on their entries. That is, given $\dnnactivation:\mathbb{R}\to\mathbb{R}$, we let
\[\dnnactivation([\mathrm{v}_{1},\dots,\mathrm{v}_{l}]):=[\dnnactivation(\mathrm{v}_{1}),\dots,\dnnactivation(\mathrm{v}_{l})].\]

\begin{definition}{\bf(Standard layer)}
    \label{def:layer}
    Let $n,m$ be positive integers. A \textit{standard layer} with activation function $\dnnactivation:\mathbb{R}\to\mathbb{R}$ is a map $L:\mathbb{R}^{m}\to\mathbb{R}^{n}$ of the form
    \[L(\mathbf{v})=\dnnactivation\left(\weight\mathbf{v}+\bias\right),\]
    where $\weight\in\mathbb{R}^{m\times n}$ and $\bias\in\mathbb{R}^{n}$ are the layer parameters, referred to as the \textit{weight matrix} and the \textit{bias vector}, respectively. The layer is said to be \textit{affine} if $\dnnactivation$ is the identity map.
\end{definition}

A classical choice for the activation function  $\dnnactivation$ is the Rectified Linear Unit (ReLU). Given a scalar input $a\in\mathbb{R}$, the latter acts as
\[\dnnactivation(a):=\max\{0,a\}.\]
Architectures based on ReLU activations are very popular, as they reproduce the same expressivity of free-knot splines \cite{daubechies2022nonlinear}. As specified below, ReLU networks are just compositions of multiple layers with ReLU activation.

\begin{definition}{\bf(ReLU network)}
\label{def:dnn}
Let $m,n$ be positive integers. We say that a map $\Phi:\mathbb{R}^{m}\to\mathbb{R}^{n}$ is a \textit{ReLU network} if it can be written as
\[\Phi=L_{\ell+1}\circ L_{\ell}\circ\dots\circ L_{1}\]
for some $\ell\ge0$ and some $L_{1},\dots,L_{\ell+1}$, where $L_{i}:\mathbb{R}^{n_{i-1}}\to\mathbb{R}^{n_{i}}$ are standard layers with ReLU activation for $i=1, \ldots, \ell$, $n_{0}:=m$, and $L_{\ell+1}:\mathbb{R}^{n_\ell}\to\mathbb{R}^{n}$ is an affine layer. The layers $L_{1},\dots, L_{\ell}$ are called \textit{hidden layers}, whereas $L_{\ell+1}$ is referred to as the \textit{output layer}. 
\end{definition}
In general, a tuple of layers $(L_{\ell+1},L_{\ell},\dots,L_{1})$ naturally defines a composite architecture of depth $\ell$ and size
\[\text{size}:=\sum_{j=1}^{\ell+1}\left(\|\weight_{j}\|_{0}+\|\bias_{j}\|_{0}\right),\]
where $\weight_{j}$ and $\bias_{j}$ are the weight matrix and the bias vector of $L_{j}$, while $\|\mathbf{A}\|_{0}$ denotes the number of nonzero entries in the tensor $\mathbf{A}.$
\\\\
We note that, in principle, a ReLU network $\Phi:\mathbb{R}^{m}\to\mathbb{R}^{n}$ may admit multiple representations, possibly referring to layer tuples of different depth and size. For instance, the map $\Phi:\mathbb{R}^{1}\to\mathbb{R}^{1}$ defined as $\Phi(a):=\dnnactivation(a)$ can be equivalently re-written as $\Phi(a)=\dnnactivation(\dnnactivation(a))$. The first representation has depth 1 and size 2, whereas the second one has depth 2 and size 3. This ambiguity comes from the fact that \textit{depth} and \textit{size} are properties of layer tuples, rather than properties of their composition. These considerations bring us to the following.

\begin{definition}{\bf(Depth and size)}
    \label{def:depthsize}
    We say that a ReLU network has depth $\le\ell$ and size $\le S$ if it can be realized through a tuple of layers with depth $\le \ell$ and size $\le S$.
\end{definition}

With this clarification, we can now continue our summary by moving to convolutional layers and, thus, Convolutional Neural Networks (CNNs). Historically, convolutional architectures were first introduced to handle time series and RGB images \cite{lecun1995convolutional}, which were commonly stored in data structures orgnanized into  \textit{channels}. For instance, in the case of 1D convolutions, CNNs are designed to accept inputs of dimension $\mathbb{R}^{m\times n}$ and return outputs of dimension $\mathbb{R}^{m'\times n'}$. Thus, they can only be connected to standard DNNs up to introducing suitable \textit{reshape} operations.

Compared to standard architectures, CNNs are more effective in handling high dimensional data, as, by leveraging their spatial structure, they can carry out complex computations with few degrees of freedom. Indeed, it can be shown that a convolutional layer operating from $\mathbb{R}^{m\times n}$ to $\mathbb{R}^{m'\times n'}$ is formally equivalent to a standard layer $\mathbb{R}^{mn}\to\mathbb{R}^{m'n'}$ whose weight matrix is sparse and contains shared entries \cite{petersen2020equivalence} (see Figure~\ref{fig:cnn} for an illustration). 
\begin{figure}
    \centering
    \includegraphics[width=\textwidth]{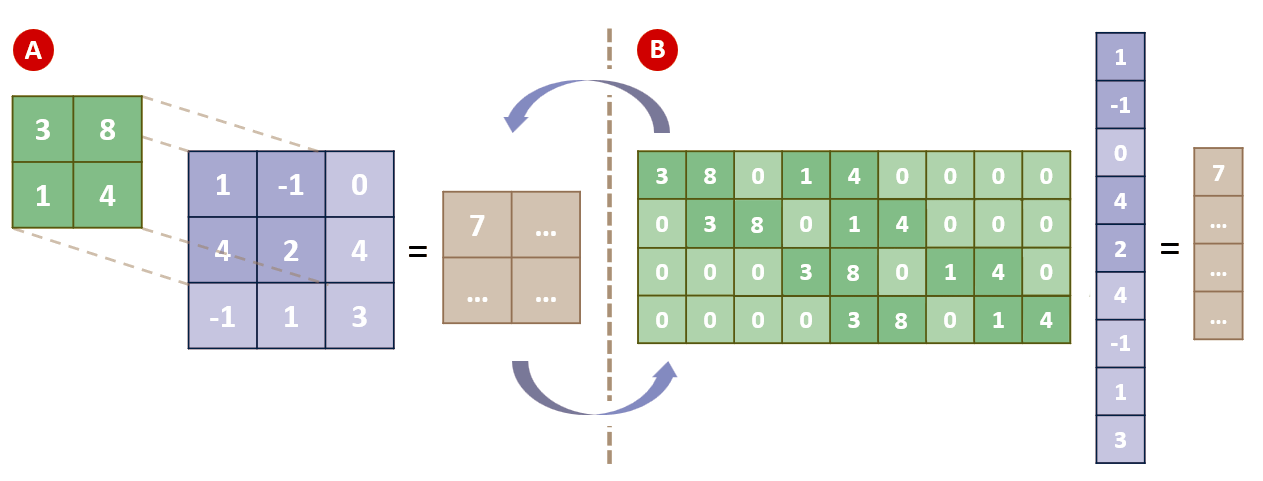}
    \caption{A 2D convolutional layer acting on a given input (simplified setting: 1 channel at input/output, no activation nor bias). The action of the convolutional layer can be visualized either in terms of a moving filter (A) or using the equivalent matrix representation (B): in both cases, despite mapping from $\mathbb{R}^{9}$ onto $\mathbb{R}^{4}$, the layer only comes with 4 learnable parameters, instead of $9 \cdot 4=36$.}
    \label{fig:cnn}
\end{figure}
\\\\
We provide a more rigorous definition of these architectures below. In what follows, we make use of the following notation, which is rather helpful when dealing with tensor objects. Given $\mathbf{A}\in\mathbb{R}^{n_{1}\times\dots\times n_{d}}$, we write $\mathbf{A}_{i_{1},\dots,i_{p}}$ for the $n_{p+1}\times\dots\times n_{d}$ subtensor obtained by fixing the first $p$ dimensions along the specified axis, where $1\le i_{j}\le n_{j}$. 

\begin{definition}{\bf(1D Convolutional layer)}
\label{def:conv}
Let $m,\review{m',s},t,d$ be positive integers and let $g$ be a common divisor of $m$ and $m'$. For any input size $n\in\mathbb{N}$, $n>0$, let
\[n_{\text{out}}:=\left\lfloor\frac{n-d(s-1)-1}{t} + 1\right\rfloor\]A 1D Convolutional layer with $m$ input channels, $m'$ output channels, grouping number $g$, kernel size $s$, stride $t$, dilation factor $d$ and activation function $\dnnactivation:\mathbb{R}\to\mathbb{R}$, is a map of the form
\[L:\mathbb{R}^{m\times n}\to\mathbb{R}^{m'\times n_{\text{out}}}\]
whose action on a given input $\mathbf{V}\in\mathbb{R}^{m\times n}$ is defined as
 \[L(\mathbf{V})_{k'}=\dnnactivation\left(\sum_{k\in \mathcal{K}}\mathbf{W}_{k',k}\otimes_{t,d}\mathbf{V}_{k}+\mathbf{B}_{k'}\right),\]
 where $1\le k'\le m'$, 
 while
 \[\mathcal{K}=\left\{\lfloor g(k'-1)/m\rfloor m/g+1,\dots,\left(\lfloor g(k'-1)/m\rfloor+1\right) m/g\right\}.\]
 Here, $\mathbf{W}\in\mathbb{R}^{m'\times (m/g)\times s}$ and $\mathbf{B}\in\mathbb{R}^{m'\times n_{\text{out}}}$ are the weight tensor and the bias matrix, respectively, whereas $\otimes_{t,d}$ is the cross-correlation operator with stride $t$ and dilation $d$. The latter is defined so that, for any $\mathbf{w}\in\mathbb{R}^{s}$ and $\mathbf{v}\in\mathbb{R}^{n}$, one has $\mathbf{w}\otimes_{t,d}\mathbf{v}\in\mathbb{R}^{n_{\text{out}}}$, where
 \[\left(\mathbf{w}\otimes_{t,d}\mathbf{x}\right)_{j}:=\sum_{i=1}^{s}w_{i}v_{(j-1)t+(i-1)d+1}.\]
\end{definition}
The default values for the stride and the dilation factor are $t=1$ and $d=1$, respectively. For this reason, with a slight abuse of notation, one says that $\Phi$ has no stride and no dilation when $t=d=1$. Similarly, we assume $g=1$ whenever the grouping number is not declared explicitly.

\begin{definition}{\bf(1D transposed convolutional layer)}
\label{def:tconv}
Let $m,\review{m',s},t,d$ be positive integers and let $g$ be a common divisor of $m$ and $m'$. For any input size $n\in\mathbb{N}$, $n>0$, let
\[n_{\text{out}}:=(n-1)t+d(s-1)+1.\]
A 1D transposed convolutional layer with $m$ input channels, $m'$ output channels, grouping number $g$, kernel size $s$, stride $t$, dilation factor $d$ and activation function $\dnnactivation:\mathbb{R}\to\mathbb{R}$, is a map of the form
\[L:\mathbb{R}^{m\times n}\to\mathbb{R}^{m'\times n_{\text{out}}}\]
whose action on a given input $\mathbf{V}\in\mathbb{R}^{m\times n}$ is defined as
\[L(\mathbf{V})_{k'}=\dnnactivation\left(\sum_{k\in\mathcal{K}}\mathbf{W}_{k,k'}\otimes_{t,d}^{\top}\mathbf{V}_{k}+\mathbf{B}_{k'}\right),\]
where $1\le k'\le m'$, while 
\[\mathcal{K}=\left\{\lfloor g(k'-1)/m\rfloor m/g+1,\dots,\left(\lfloor g(k'-1)/m\rfloor+1\right) m/g\right\}.\]
Here, $\mathbf{W}\in\mathbb{R}^{(m/g)\times m'\times s}$ and $\mathbf{B}\in\mathbb{R}^{m'\times n_{\text{out}}}$ are the weight tensor and the bias matrix, respectively, whereas $\otimes_{t,d}^{\top}$ is the transposed cross-correlation operator with stride $t$ and dilation $d$. The latter is defined so that, for any $\mathbf{w}\in\mathbb{R}^{s}$ and $\mathbf{v}\in\mathbb{R}^{n}$, one has $\mathbf{w}\otimes_{t,d}^{\top}\mathbf{v}\in\mathbb{R}^{n_{\text{out}}},$ where
\[\left(\mathbf{w}\otimes_{t,d}^{\top}\mathbf{v}\right)_{j}:=\sum_{i\in\mathcal{I}}w_{\left\lfloor(i-1)t/d+(1-j)/d\right\rfloor+1}v_{i},\]
with $\mathcal{I}=\left\{\left\lfloor\frac{j-1}{t}+1\right\rfloor,\dots,\left\lfloor\frac{(s-1)d+j-1}{t}+1\right\rfloor\right\}.$
\end{definition}

As we mentioned, it is also useful to define reshaping operations. We provide a rigorous definition below.

\begin{definition}{\bf(Reshape)}
\label{def:reshape}
Let $m$ and $n$ be positive integers. Let $R_{m,n}:\mathbb{R}^{mn}\to\mathbb{R}^{m\times n}$ be the bijective linear map defined as
\[R_{m,n}: x\mapsto\left[
\begin{array}{lll}
    x_{1} & \dots &  x_{n}\\
    x_{n+1} & \dots & x_{2n}\\
    \dots & \dots & \dots\\
    x_{(m-1)n+1} & \dots & x_{mn}
\end{array}
\right].\]
The map $R_{m,n}$ and its inverse, $R_{m,n}^{-1}$, are called \emph{reshape} operations.
\end{definition}

\begin{definition}{\bf(Convolutional Neural Network)}
    \label{def:cnn}
    We say that a map $\Psi$ is a \emph{Convolutional Neural Network (CNN)} if it can be realized as the composition of (transposed) convolutional layers and reshape operations.
\end{definition}

It is important to note that, by including reshape operations, CNNs can accept both vectors and matrices. In fact, any CNN $\tilde{\Psi}:\mathbb{R}^{m\times n}\to\mathbb{R}^{m'\times n'}$ comes with its vectorized counterpart $\Psi:=R_{m',n'}^{-1}\circ\tilde{\Psi}\circ R_{m,n}$. Note in fact that, although $\Psi$ operates on vectors, it can be considered a CNN according to Definition~\ref{def:cnn}.
\\\\
The concepts of depth and size can be easily generalized to CNNs in the natural way. We mention that, in doing so, reshape modules are typically ignored. Finally, in what follows, \review{we will say that a CNN has} at most $q$ channels per layer if it can be realized without relying on convolutional layers that have more than $q$ channels (either at input or output). Similarly, when stating that a CNN has depth $\le\ell$, size $\le S$, number of channels per layer $\le q$, \review{and kernel size per layer $\le K$,} we intend that there exists a representation of such architecture satisfying all those requirements simultaneously.

\section{Main result}
\label{sec:main}
We are now ready to present our main result. However, before coming to the actual statement of the Theorem~\ref{theorem:cnn}, it is worth recapping the general context.
Let $p\in\mathbb{N}$, $p\ge1$, $\Omega=(0,1)$ and $\paramdomain=[-1,1]^{p}$. Let $\{x_{1}<x_{2}<\dots<x_{N_{h}}\}\subset\overline{\Omega}$ be a fixed spatial grid, and let $\probability$ be the uniform probability distribution over $\paramdomain$. DL-ROMs aim at approximating the map
\[\paramdomain\ni\mub\mapsto[u_{\mub}(x_{1}),\dots,u_{\mub}(x_{N_{h}})]^\top\in\mathbb{R}^{N_{h}},\]
where $u_{\mub}:=\operator(\mub)$ is the solution to some parametrized PDE, with $\operator:\paramdomain\to H^{s}(\Omega)$ taking values in a suitable Sobolev space, $s\ge1.$

In the DL-ROM paradigm, such approximation is provided by a neural network architecture \[\Phi:\mathbb{R}^{p}\to\mathbb{R}^{N_{h}},\] obtained via the composition of a reduced module, $\phi:\mathbb{R}^{p}\to\mathbb{R}^{m}$, and a decoder module, $\Psi:\mathbb{R}^{m}\to\mathbb{R}^{N_{h}}$. During training, these architectures are supplemented by an auxiliary encoder module, $\Psi':\mathbb{R}^{N_{h}}\to\mathbb{R}^{m}$, which effectively turns the intermediate state space, $\mathbb{R}^{m}$, onto a \textit{latent} space. In fact, the idea is that
\[\Phi(\mub):=\Psi(\phi(\mub))\approx[u_{\mub}(x_{1}),\dots,u_{\mub}(x_{N_{h}})]^\top,\]
while, at the same time, $\phi(\mub)\approx\Psi'([u_{\mub}(x_{1}),\dots,u_{\mub}(x_{N_{h}})]^\top).$
\\\\
Recently, the theory underlying DL-ROMs has evolved significantly, see, e.g. \cite{brivio2023error, franco2023latent, franco2023approximation, franco2023deep}. However, practical insights on the implementation and training of these architectures are far from being exhaustive. For this reason, domain practitioners often rely on suitable rules of thumb, deduced from empirical evidence. Some of these include:

\begin{itemize}
    \item [i)] in practice, althought certain studies suggest otherwise \cite{mishra2021enhancing}, ReLU networks are particularly expressive compared to other architectures that rely on different nonlinearities, such as the sigmoidal activation \cite{oostwal2021hidden}. Thus, the ReLU activation, \review{together with its smooth variants (softplus, SeLU, GeLU, etc.),} can be a good choice when constructing DL-ROMs. \review{Here, for the sake of simplicity, we shall focus on the ``standard'' ReLU};
    \item [ii)] convolutional layers can significantly enhance the performances of the decoder module. In fact, given the high dimensionality of the output, classical layers would be prohibitive to train \cite{lee2020model, mucke2021reduced};
    \item [iii)] the number of convolutional layers in the decoder module should be proportional to the resolution of the spatial grid \cite{franco2023approximation};
    \item [iv)] the encoder block does not need to be as complex as the decoder, nor it benefits as much from the use of convolutional layers \cite{franco2023deep};
    \item [v)] ideally, it should be possible to use the same autoencoder for different problems simultaneously \cite{masoumi-verki2023use, sun2018deep}, without the need for re-training (a practice also known as \emph{transfer learning}). In fact, classical compression techniques based on, e.g., Fourier transform, or wavelets, are somewhat universal. Similarly, we should be able to find problem-agnostic autoencoders that do not require a specific training routine;
    \item [vi)] the reduced map should be trained by taking the encoder outputs as a ground truth reference, i.e., by minimizing
    \[\frac{1}{N}\sum_{i=1}^{N}\|\phi(\mub_{i})-\Psi'([u_{\mub_{i}}(x_{1}),\dots,u_{\mub_{i}}(x_{N_{h}})]^\top)\|_{2}^{2},\]
    where $\{\mub_i\}_{i=1}^{N}\subset\paramdomain$ is an independent identically distributed (i.i.d.) random sample, generated according to $\probability$;
    \item [vii)] to avoid overfitting and ensure a proper generalization, DL-ROMs can benefit from suitable regularization strategies \cite{fresca2021comprehensive}, especially at the latent level \cite{romor2022non-linear};
\end{itemize}

As we shall see in a moment, by embedding convolutional neural networks within the novel framework of \textit{practical existence theorems}, we can finally derive a comprehensive theory supporting these heuristics. We report our main result, Theorem \ref{theorem:cnn}, right below. For the sake of better readability, the proof is postponed to Section \ref{sec:proof}. \review{In what follows, given random variable $X$, we shall write $\mathbb{E}^{1/2}[X]$ as a short-hand notation for $(\mathbb{E}[X])^{1/2}$.}

\begin{theorem} \label{theorem:cnn}
There are universal constants $c_{0},c_{1},c_{2},c_{3},c_{4}>0$ such that the following holds.
Let $p\in\mathbb{N}$, $p\ge1$, and $\epsilon,\gamma>0$. Let $\probability$ be the uniform probability distribution over $\paramdomain:=[-1,1]^{p}$. Let $\Omega=(0,1)$ and \[\operator:\paramdomain\ni\mub\to u_{\mub}\in H^{s}(\Omega)\] be a (nonlinear) map belonging to $\ha_{\gamma,\epsilon,s}(\paramdomain)$, where $s\ge1$ (see Definition~\ref{def:hidden-sobolev}). Fix a training size $N\ge1$ and a probability of failure $0<\varepsilon<1.$ Define
\begin{align*}
    &\tilde{N}:= N\cdot(c_{0}\log(2N)(\log(2N)\min\{\log(2N)+p,\;\log(2N)\log(2p))\}+\log(1/\varepsilon))^{-1}\\
    &\Delta := \min\left\{2^{p/2+1}\tilde{N}^{3/2},\;e^{2}(\tilde{N}/2^{p})^{1+\frac{1}{2}\log_{2}p},\;\frac{\tilde{N}^{1/2}(\log\tilde{N} + (p+1)\log2)^{p-1}}{2^{p/2-1}(p-1)!}\right\}.
\end{align*} Let $\{x_{j}\}_{j=1}^{N_{h}}\subset\overline{\Omega}$ be an equispaced grid of stepsize $h=2^{-k}$ for some $k\in\mathbb{N}$. 
Fix a latent dimension $m\ge1$ and let $\tilde{m}:=2m+1$. Then, there exist
\begin{itemize}
    \item [a)] a class of ReLU networks $\dense$ from $\mathbb{R}^{p}\to\mathbb{R}^{\tilde{m}}$ with
    \begin{align*}
         \;\;\;\;\;\;&\emph{depth}\;\le\;c_{1}(1+p\log p)(1+\log\tilde{N})\left((\tilde{N}/2^{p})^{1/2}+\log(\Delta)+\gamma\tilde{N}^{1/(2p)}\right),
         \\ \;\;\;\;\;\;&\emph{size}\;\le\;c_{2}p\left(p\tilde{N}/2^{p}+\left((\tilde{N}/2^{p})^{1/2}+p\Delta\right)\left(\log(\tilde{N}\Delta)+\gamma\tilde{N}^{1/(2p)}\right)\right)+\tilde{m}\Delta,
    \end{align*}
    \item [b)] a latent regularization function $\regularizer:\dense\to[0,+\infty)$, equivalent to a certain norm of the trainable parameters, and a regularization parameter $\lambda=\lambda(\tilde{N},p)$,\vspace{0.5em}
    \item [\review{c)}] a ReLU convolutional neural network, $\Psi:\mathbb{R}^{\tilde{m}}\to\mathbb{R}^{N_{h}}$, whose architecture only depends on 
    \review{$\operator$ through $s$,}
    with
         \[\emph{depth}\;\le c_{3}\log(1/h),
         \quad\quad\emph{size}\;\le c_{3}m\log(1/h),\]
         \[\emph{channels per layer}\;\le8m,\review{
         \quad\quad\emph{kernel size per layer}\;\le2,}\]
    acting as a decoder,\vspace{0.5em}
    \item [\review{d)}] a ReLU network, $\Psi':\mathbb{R}^{N_{h}}\to\mathbb{R}^{\tilde{m}}$, whose architecture only depends on 
    \review{
    $\operator$ through $s$ and $\|\operator\|_{L^{\infty}(\polyellipse_{\brho}, H^{s}(\Omega))}$,}
    operating as an encoder,\vspace{0.5em}
\end{itemize}
such that the following holds with probability $1-\varepsilon$. Let $\{\mub_{i}\}_{i=1}^{N}$ be an i.i.d.\ random sample, uniformly drawn from the parameter space $\paramdomain$. Denote by $P$ the function-to-grid operator, $P:u\mapsto[u(x_{1}),\dots,u(x_{N_{h}})]$. Every minimizer $\hat{\phi}\in\dense$ of
\begin{equation}
    \label{eq:minim}
    \min_{\phi\in\dense}\;\sqrt{\frac{1}{N}\sum_{i=1}^{N}\|\phi(\mub_{i})-\Psi'(Pu_{\mub_{i}})\|_{2}^{2}}+\lambda\regularizer(\phi)
\end{equation}
satisfies
\begin{multline}
    \label{eq:mainresult}
    \expe_{\mub\sim\probability}^{1/2}\left[\;\sup_{j=1,\dots,N_{h}}|u_{\mub}(x_{j})-\Psi_{j}(\hat{\phi}(\mub))|^{2}\right]\le\\\le c_{4}\left(\sqrt{m}e^{-\frac{1}{\sqrt{2}}\gamma\tilde{N}^{1/(2p)}}+\sqrt{\frac{2m^{1-2s}}{2s-1}}\right)\|\operator\|_{L^{\infty}(\polyellipse_{\brho}, H^{s}(\Omega))},
\end{multline}
for all $\tilde{N}\ge N_{0}$, where $N_{0}=N_{0}(\epsilon,\operator,m,p,s)$ and $\polyellipse_{\brho}$ is as in Definition~\ref{def:hidden}.
\end{theorem}

Theorem~\ref{theorem:cnn} has multiple implications. First of all, it shows clearly how the different properties of the problem affect the design of the neural network architectures. For instance, the complexity of the decoder, both in terms of depth and size, scales logarithmically with the grid resolution, $h$. In contrast, the reduced network, $\phi$, does not depend on $h$, but on $p$. Notably, thanks to the regularity of the parameter-to-solution map, the reduced network is 
only mildly affected by the curse of dimensionality: its size grows at most quadratically in $p$ \review{(up to logarithmic factors)}. 

Another interesting fact concerns the latent dimension, $m$, which directly appears in the error bound \eqref{eq:mainresult}. On the hand, increasing $m$ can improve the accuracy of the model (at a rate that depends on the smoothness of the PDE solutions, $s$). However, in order to generalize properly, DL-ROMs with a larger latent space necessitate of more training data, as clearly depicted by the term $\sqrt{m}\exp(-\tilde{N}^{1/(2p)}/\sqrt{2})$. \review{More specifically, in order to achieve a prescribed target accuracy level $\tau >0$ it is sufficient for $m$ to scale polynomially in $1/\tau$ and for $\tilde{N}$ to scale polynomially in $\log(1/\tau)$ (this can be seen by bounding the two main terms in the right-hand side of \eqref{eq:mainresult} from above with $\tau$ and rearranging the corresponding inequalities).} 

In general, Theorem~\ref{theorem:cnn} shows that the error of a trained DL-ROM can be bounded by two terms: a \textit{sampling error}, which---asymptotically---decays exponentially with respect to the training set size, and an \textit{approximation error}, driven by the architecture design and the output smoothness. Interestingly, our result confirms most of the heuristics adopted by domain practitioners: from the use of ReLU networks and convolutional autoencoders, to the introduction of latent regularization techniques. On this note, we also observe that Theorem~\ref{theorem:cnn} implicitly supports the use of \textit{transfer learning}. In fact, \review{looking back at the proof}, our result suggests that, for a fixed degree of smoothness, there exists a universal autoencoder performing equivalently well for all operators $\operator$ (up to a norm factor). In practice, such autoencoder could be initialized and trained \emph{a priori}, by relying on synthetic data. 

\review{The regularizer $\mathcal{R}$ is an important component of Theorem~\ref{theorem:cnn}. Recalling the notation introduced in Definition~\ref{def:dnn}, the latter can be explicitly characterized using the matrix $\ell^{2,1}$-norm as 
$$
\mathcal{R}(\phi) 
= \|\mathbf{W}_{\ell+1}\|_{2,1} 
:= \sum_{j = 1}^{n_{\ell}} \|\mathbf{W}_{\ell+1} \mathbf{e}_j \|_{2},
$$
where $\mathbf{W}_{\ell+1} \in \mathbb{R}^{\tilde{m} \times n_{\ell}}$ is the weight matrix associated with the last (linear) layer of $\phi$, while $\mathbf{e}_{j}\in\mathbb{R}^{n_{l}}$ is the $j$th vector of the canonical basis. The presence of this regularization term is essential to prove the practical existence theorem in \cite{adcock2022deep} (upon which Theorem~\ref{theorem:cnn} relies). In fact, it allows one to rigorously connect deep neural network training with sparse polynomial approximation via compressed sensing. For a more detailed discussion, we refer the reader to \cite{adcock2024learning}.}

\review{
In relation to this, an inspection of the proof of Theorem~\ref{theorem:cnn} reveals that 
$\hat{\phi}$ consists mostly of sparsely connected layers. 
%
This sparsity is further promoted by the regularizer $\mathcal{R}$, whose action naturally favors \textit{compressibility} (i.e., approximate sparsity) 
of the last layer's weights; keeping only the absolute largest weights of this layer would yield improved network size bounds (see \cite[Section~9.4]{adcock2024learning} for further details). These observations suggests that adopting \emph{network pruning} in this context might be an effective strategy (see \cite{frankle2019lottery}); notably, this is coherent with recent empirical evidence in the reduced order modeling literature, see, e.g., \cite{franco2023mesh-informed, hernandez2021deep}.} 

\review{Clearly, despite offering several insights, Theorem~\ref{theorem:cnn} comes with its own limitations: we provide a detailed discussion on the matter in Section~\ref{sec:conclusions}.}

\subsection{Application to a parametric diffusion model}
\label{sec:param_diff}
To showcase the applicability of Theorem~\ref{theorem:cnn}, we consider a parametric diffusion equation with affine parametric dependence on the diffusion term. This model is often used as a case study in the parametric PDE literature (see, e.g.,  \cite[Chapter 4]{adcock2022sparse}, \cite{cohen2015approximation, marcati2023exponential}, and references therein). 

We consider the physical domain $\Omega = (0,1)$, a forcing term $F \in H^{-1}(\Omega)$, and functions $a_0 \in L^\infty(\Omega)$, $\{\psi_j\}_{j=1}^p \subset L^\infty(\Omega)$ defining an affine parametric diffusion term
\begin{equation}
    a_{\mub}(x) = a_0(x) + \sum_{j=1}^p \mu_j \psi_j(x), \quad x\in\Omega, \; \mub \in \paramdomain.
\end{equation}
Then, we consider the following parametric weak problem: for any $\mub\in \paramdomain$, find $u_{\mub}\in H_0^1(\Omega)$ such that
\begin{equation}
\label{eq:weak_diffusion}
\int_{\Omega} a_{\mub} \frac{d u_{\mub}}{dx} \frac{dv}{dx} \, dx = \int_\Omega Fv \, dx, 
\quad \forall v \in H_0^1(\Omega).
\end{equation}
In order to apply Theorem~\ref{theorem:cnn}, we need to (i) find sufficient conditions 
ensuring that the map $\operator:\mub \mapsto u_{\mub}$ belongs to $\mathcal{HA}_{\gamma, \epsilon, 1}(\paramdomain)$ and (ii) estimate 
$\|\operator\|_{L^\infty(\polyellipse_{\brho}, H^1(\Omega))}$. 

First, we assume the parametric problem \eqref{eq:weak_diffusion} to be \emph{uniformly elliptic}, i.e., such that
\begin{equation}
\label{eq:unif_ellip}
\sum_{k=1}^p |\psi_k(x)| \leq a_0(x) - r,
\end{equation}
for some $r>0$. This implies, in particular, that $\essinf_{x\in\Omega} a_{\mub}(x) \geq r$ for every $\mub\in \Omega$ (and, hence, that \eqref{eq:weak_diffusion} is elliptic for every fixed $\mub \in \paramdomain$). In addition, for some fixed 
$\gamma,\epsilon > 0$ and $\xi >0$ we assume the functions $\{\psi_k\}_{k=1}^p$ to be such that
\begin{equation}
\label{eq:assumption_Bernstein_diffusion}
\sum_{k=1}^p \left(\frac{\rho_k + \rho_k^{-1}}{2}-1\right)\|\psi_k\|_{L^{\infty}} \leq \xi,
\end{equation}
for every $\brho$ satisfying condition \eqref{eq:condition_rho}.
In this setting, \cite[Proposition 4.9]{adcock2022sparse} immediately implies that $\operator \in \mathcal{HA}_{\gamma,\epsilon,1}$. Note that a holomorphic extension of $\operator$ to $\polyellipse_{\brho}$ is the map $\zetab \mapsto u_{\zetab}$, where $u_{\zetab}$ is the (thanks  to uniform ellipticity, unique) solution to the weak problem associated with the complex-valued diffusion coefficient 
\[
a_{\zetab}(x) = a_0(x) + \sum_{j=1}^p \zeta_j \psi_j(x), \quad x\in\Omega, \; \zetab \in \polyellipse_{\brho}.
\]
In addition, we see that
\[
\|\operator\|_{L^{\infty}(\polyellipse_{\brho}, H^1(\Omega))}
\leq \sqrt{1+\frac{1}{\pi^2}} \cdot
\|\operator\|_{L^{\infty}(\polyellipse_{\brho}, H^1_0(\Omega))} 
\leq \sqrt{1+\frac{1}{\pi^2}} \cdot \frac{\|F\|_{H^{-1}(\Omega)}}{r-\xi},
\]
where the first inequality hinges on the Poincar\'e inequality, while the second one is a consequence of \cite[Proposition 4.9]{adcock2022sparse}. Here, $H_0^1(\Omega)$ is equipped with its classical energy (semi)norm \[\|u\|_{H^1_0}:=\sqrt{\int_{\Omega}\left|\frac{d u}{dx}\right|^{2}dx}.\]
We are then allowed to apply Theorem~\ref{theorem:cnn} to problem \eqref{eq:weak_diffusion}, in which case the error bound \eqref{eq:mainresult} reads
\begin{multline*}
\expe_{\mub\sim\probability}^{1/2}\left[\;\sup_{j=1,\dots,N_{h}}|u_{\mub}(x_{j})-\Psi_{j}(\hat{\phi}(\mub))|^{2}\right]\le\\\le c \left(\sqrt{m}e^{-\frac{1}{\sqrt{2}}\gamma\tilde{N}^{1/(2p)}}+\sqrt{\frac{2}{m}}\right)\frac{\|F\|_{H^{-1}(\Omega)}}{r-\xi},
\end{multline*}
for some universal constant $c>0$.
\\\\
We conclude by noting that Theorem~\ref{theorem:cnn} could also be applied to problem \eqref{eq:weak_diffusion} for $s>1$. Indeed, if $F \in H^{s-2}(\Omega)$ and $a_{\mub} \in C^{s-1}(\Omega)$, then standard regularity theory results for PDEs imply that $u_{\mub}\in H^s(\Omega)$ (see, e.g., \cite[Section 6.3, Theorem~2]{evans2022partial}). However, finding precise sufficient conditions on the parametric coefficient $a_{\mub}$ able to ensure that $\operator \in \mathcal{HA}_{\gamma,\epsilon,s}$ and bounding $\|\operator\|_{L^{\infty}(\polyellipse_{\brho}, H^s(\Omega))}$ requires an extension of \cite[Proposition 4.9]{adcock2022sparse} and a careful analysis that is outside the scope of this paper.




\section{Proof of Theorem~\ref{theorem:cnn}}
\label{sec:proof}
We subdivide the proof into several steps. In particular, we shall state, and prove, a few claims that eventually lead to the full proof. Before diving into the details, we recall the definition of \textit{operator norm}. Given a linear map $T:(\mathscr{A},\|\cdot\|_{\mathscr{A}})\to(\mathscr{B},\|\cdot\|_{\mathscr{B}})$ between two normed spaces, we set
\[\opnorm{T}:=\sup_{\substack{a\in \mathscr{A}\\\|a\|_{\mathscr{A}=1}}}\|Ta\|_{\mathscr{B}}.\]
Equivalently, due to linearity, $\opnorm{T}$ is nothing but the (best) Lipschitz constant of $T$.
\\\\
\textbf{Step 1.} \textit{Without loss of generality, the function-to-grid operator, $P$, can be assumed to be injective over $\operator(\paramdomain)\subset H^{s}(\Omega).$}
\begin{proof}
    Assume that we are able to prove Theorem~\ref{theorem:cnn} whenever $P$ is injective over $\operator(\paramdomain)\subset H^{s}(\Omega)$.
    Let us now consider an operator $\tilde{\operator}$, satisfying all the hypotheses of the Theorem, but for which $P$ is not injective over the image set $\tilde{\operator}(\paramdomain).$ Then, the idea is to exploit the following Lemma, which, essentially, is just a re-writing of \cite[Theorem 5.1]{de1966splines}.
    \begin{lemma}
        \label{lemma:splines}
        Let $\Omega$, $s$ and $x_{1},\dots,x_{N_{h}}$, be as in Theorem~\ref{theorem:cnn}. There exists a bounded linear operator $Q: H^{s}(\Omega)\to H^{s}(\Omega)$ such that
        \begin{itemize}
            \item[i)] $(Qf)(x_{j})=f(x_{j})$ for $f\in H^{s}(\Omega)$ and all $j=1,\dots, N_{h};$\vspace{0.25cm}
            \item[ii)] if $f,g\in H^{s}(\Omega)$ and $f(x_{j})=g(x_{j})$ for all $j=1,\dots, N_{h}$, then \[\|Qf\|_{H^{s}(\Omega)}\le \|g\|_{H^{s}(\Omega)}.\]
        \end{itemize}
    \end{lemma}
    In practice, $Q$ is a projection operator that maps $H^{s}(\Omega)$ onto a suitable subspace of smooth splines. Most importantly, $Q$ acts as an interpolator with minimum norm: see (i) and (ii), respectively. Furthermore, it is straightforward to see that $P$ is injective over $Q(H^{s}(\Omega)).$ In fact, by letting $g\equiv0$ in (ii), and by exploiting (i), we see that
    \[P(Qf)=0\implies Pf=0=Pg\implies\|Qf\|_{H^{s}(\Omega)}\le\|g\|_{H^{s}(\Omega)}=0\implies Qf\equiv0.\]

    With this in mind, let $Q$ be as in Lemma~\ref{lemma:splines}, and let $\tilde{\operator}_{Q}:=Q\circ\tilde{\operator}$. Since $Q$ is both linear and continuous, it is holomorphic, and, furthermore, \[\tilde{\operator}\in\ha_{\gamma,\epsilon,s}(\paramdomain)\implies\tilde{\operator}_{Q}\in\ha_{\gamma,\epsilon,s}(\paramdomain).\]   
    In particular, since $\tilde{\operator}_{Q}$ satisfies all the properties in Theorem~\ref{theorem:cnn} and $P$ is injective over $\tilde{\operator}_{Q}(\paramdomain)\subseteq Q(H^{s}(\Omega))$, we are allowed to invoke Theorem~\ref{theorem:cnn} with $\operator:=\tilde{\operator}_{Q}$, thus obtaining the error bound in Eq.~\eqref{eq:mainresult} (recall that we assumed the Theorem to hold true whenever the additional hypothesis of injectivity is satisfied). However, since $(\tilde{\operator}_{Q}\mub)(x_{j})=(\tilde{\operator}\mub)(x_{j})=u_{\mub}(x_{j})$ for all $j=1,\dots, N_{h}$, and \[\|\tilde{\operator}_{Q}\|_{L^{\infty}(\polyellipse, H^{s}(\Omega))}\le \|\tilde{\operator}\|_{L^{\infty}(\polyellipse, H^{s}(\Omega))}\]
    due to (ii), it is evident that \eqref{eq:mainresult} also holds for $\operator:=\tilde{\operator}$, thus proving our claim. 
\end{proof}

\noindent\textbf{Step 2.} \textit{There exists a linear operator $T:H^{s}(\Omega)\to\mathbb{R}^{\tilde{m}}$ and a CNN $\Psi:\mathbb{R}^{\tilde{m}}\to\mathbb{R}^{N_{h}}$ satisfying (d), such that}
\begin{equation}
\label{eq:autoencoder}
\sup_{j=1,\dots,N_{h}}|u(x_{j})-\Psi_{j}(Tu)|\le \sqrt{\frac{2}{2s-1}}m^{1/2-s}\|u\|_{H^{s}(\Omega)}\quad\forall u\in \operator(\paramdomain)\subset H^{s}(\Omega),\end{equation}
and $\opnorm{T}\le2.$
\begin{proof}
By \cite[Theorem 1]{franco2023approximation} there exists a continuous linear operator $T:H^{s}(\Omega)\to\mathbb{C}^{2m+1}$ and a linear CNN (no activations nor biases at any level) $\Psi:\mathbb{C}^{2m+1}\to\mathbb{R}^{N_{h}}$, whose depth, size and number of channels satisfy the complexity bounds in (d), such that \eqref{eq:autoencoder} holds\footnote{\review{\cite[Theorem 1]{franco2023approximation} does not mention the bound on the kernel size explicitly; however, this is a direct consequence of \cite[Lemma 3]{franco2023approximation}, upon which the previous Theorem is built.}}. The operator $T$ only depends on $s$, and its operator norm can be bounded as $\opnorm{T}\le2$: we refer the reader to the Appendix, Lemma~\ref{lemma:T}, for a rigorous description of $T$ and its properties.

We note, however, that we cannot readily use such $T$ and $\Psi$, as they take values (respectively, inputs) in $\mathbb{C}^{\tilde{m}}\cong\mathbb{R}^{2\tilde{m}}.$ To fix this, for any $k\in\mathbb{N}$, let $B:\mathbb{C}^{\tilde{m}}\to\mathbb{R}^{\tilde{m}}$ be the linear map
\[\review{B\left([a_{-m}+ib_{-m},\dots,a_{0}+ib_{0} ,\dots,a_{m}+ib_{m}]\right)= [a_{0},a_{1},b_{1},\dots,a_{m},b_{m}]}\]
(recall that $\tilde{m}=2m+1$), and let $B^{\dagger}:\mathbb{R}^{k}\to\mathbb{C}^{k}$ be its pseudo-inverse, acting as
\[\review{B^{\dagger}\left([a_{0},a_{1},b_{1},\dots,a_{m},b_{m}]\right) = [a_{m}-ib_{m},\dots,a_{0},\dots,a_{m}+ib_{m}].}\]
By diving deeper into the definition of $T$, cf. Eq.~\eqref{eq:Tdef} in the Appendix, we see that for all $u\in H^{s}(\Omega)$ the image vector \[Tu=[z_{-m},\dots,z_{0},\dots,z_{m}]\in\mathbb{C}^{\tilde{m}}\] satisfies $z_{0}\in\mathbb{R}$ and $z_{k}=\overline{z_{-k}}$
for all $k\in\{1,\dots,m\}$. Consequently, it is straightforward to see that
\[\Psi(B^{\dagger}BTu) = \Psi(T u)\]
for all $u\in H^{s}(\Omega)$. In light of this, we are allowed to replace $T$ with $B\circ T$ and $\Psi$ with $\Psi\circ B^{\dagger}$, so that the two maps operate on the right spaces (i.e., $\mathbb{R}^{\tilde{m}}$ and not $\mathbb{C}^{\tilde{m}}$). In this concern, note also that $\opnorm{B}\le 1$: in particular, the bound on the operator norm is preserved.
To keep the notation lighter, the presence of $B$ and $B^{\dagger}$ will be omitted.
\\\\
Note: with this construction, $\Psi$ is linear. However, since $T(\operator(\paramdomain))$ is compact, we can easily turn $\Psi$ onto a ReLU CNN (without changing its outputs) by including suitable biases within the layers of the architecture. We refer to Lemma~\ref{lemma:relu} and Corollary~\ref{corollary:relu} in the Appendix for a detailed explanation. Once again, in order to simply the notation, we shall directly assume $\Psi$ to be a ReLU CNN and avoid the introduction of auxiliary architectures.
\end{proof}

\noindent\textbf{Step 3.} \textit{For every $\delta >0$, there exists a ReLU encoder $\Psi':\mathbb{R}^{N_{h}}\to\mathbb{R}^{\tilde{m}}$ such that}
\begin{equation}
\label{eq:deltaclose}\sup_{\review{\mub\in \paramdomain}}\|Tu_{\mub}-\Psi'(Pu_{\mub})\|_{2}<\delta.\end{equation}
\begin{proof}
In light of Step 1, we assume $P$ to be injective over $\operator(\paramdomain).$  Since the latter is compact (recall that $\operator$ is continuous) and $P$ is continuous, this suffices to show that $P$ admits a continuous inverse \[P^{-1}:P(\operator(\paramdomain))\to \operator(\paramdomain),\]
which we may readily extend to a broader map from $\mathbb{R}^{N_{h}}$ onto $H^{s}(\Omega)$ (see, e.g., Dugundji's extension Theorem \cite{dugundji1951extension}): with little abuse of notation, we shall still denote this extension by $P^{-1}$.
Let $E:=T\circ P^{-1}$, so that  $E:\mathbb{R}^{N_{h}}\to\mathbb{R}^{\tilde{m}}$, and fix any tolerance $\delta>0$. Then, there exists a ReLU network $\Psi':\mathbb{R}^{N_{h}}\to\mathbb{R}^{\tilde{m}}$ such that
\begin{equation}
\label{eq:Epsiprime}
\sup_{\mathbf{v}\in P(\operator(\paramdomain))}\|E(\mathbf{v})-\Psi'(\mathbf{v})\|_{2}<\delta.\end{equation}
The existence of such $\Psi'$ is guaranteed by the compactness of $P(\operator(\paramdomain))$ and by the continuity of $E$, as ReLU networks are known to be dense in the space of continuous maps over compact subsets \cite{hornik1991approximation}. Since, by definition, we also have
\begin{equation}
\label{eq:discrete to continuous}
E(Pu_{\mub})=TP^{-1}P(u_{\mub})=Tu_{\mub},
\end{equation}
for all $\mub\in\paramdomain$, it is clear that \eqref{eq:Epsiprime} is nothing but \eqref{eq:deltaclose}.
\end{proof}

\begin{remark}In what follows, we let $\regularizer:\dense\to[0,+\infty)$ be the regularization functional in \cite[Theorem 5]{adcock2022deep}, so that, for any $\phi\in\dense$, the penalty term $\regularizer(\phi)$ corresponds to the $\ell^{1}$ norm of the weights in the output layer of the network $\phi$.
\end{remark}

\noindent\textbf{Step 4.} \textit{Having fixed any $\delta>0$ and $\Psi'$ as in Step 3, for every rescaling factor $\eta>0$ one has}
\begin{multline}
    \label{eq:rescaling0}
    \argmin_{\phi\in\dense}\;\sqrt{\frac{1}{N}\sum_{i=1}^{N}\|\phi(\mub_{i})-\Psi'(Pu_{\mub_{i}})\|^{2}}+\lambda\regularizer(\phi)
    =\\=\frac{1}{\eta}\cdot\argmin_{\phi\in\dense}\;\sqrt{\frac{1}{N}\sum_{i=1}^{N}\|\phi(\mub_{i})-\eta \Psi'(Pu_{\mub_{i}})\|^{2}}+\lambda\regularizer(\phi)
    .
\end{multline}
\begin{proof}
     By definition, $\regularizer(\eta \phi)=|\eta|\regularizer(\phi)$ for all $\eta\in\mathbb{R}$. In fact, $\phi\in\dense\implies \eta\phi\in\dense$, as the latter is easily obtained by multiplying all the terminal weights in $\phi$ by the scalar value $\eta$. Then, it is straightforward to see that, for all $\eta>0$,
\begin{multline}
    \label{eq:rescaling}
    \argmin_{\phi\in\dense}\;\sqrt{\frac{1}{N}\sum_{i=1}^{N}\|\phi(\mub_{i})-\Psi'(Pu_{\mub_{i}})\|^{2}}+\lambda\regularizer(\phi)=\\=\argmin_{\phi\in\dense}\;\sqrt{\frac{1}{N}\sum_{i=1}^{N}\|\eta\phi(\mub_{i})-\eta \Psi'(Pu_{\mub_{i}})\|^{2}}+\lambda\regularizer(\eta\phi)=\\
    =\frac{1}{\eta}\cdot\argmin_{\phi\in\dense}\;\sqrt{\frac{1}{N}\sum_{i=1}^{N}\|\phi(\mub_{i})-\eta \Psi'(Pu_{\mub_{i}})\|^{2}}+\lambda\regularizer(\phi)
    .
\end{multline}
\end{proof}

\noindent
\textbf{Step 5.} 
\textit{\review{Let $\eta^{*}:=(4\|\operator\|_{L^{\infty}(\polyellipse, H^{s}(\Omega))})^{-1}$. For every $\delta>0$, and a corresponding choice of $\Psi'$, one has}}
\begin{equation}
    \label{eq:practicalbound0}
    \mathbb{E}_{\mub\sim\probability}^{1/2}\|Tu_{\mub}-\hat{\phi}(\mub)\|^{2}_{2}\le \review{\eta_{*}^{-1}c_{4}\exp\left(-\frac{1}{\sqrt{2}}\gamma\tilde{N}^{1/(2p)}\right)+c_4\delta},
\end{equation}
\textit{where $\hat{\phi}\in\dense$ is any minimizer of \eqref{eq:minim}.}
\begin{proof}
    For any $\eta>0$, let us consider the rescaled minimization problem in Step 4, and let
\begin{equation}
\label{eq:rescaled}
\hat{\phi}_{\eta}:=\argmin_{\phi\in\dense}\;\sqrt{\frac{1}{N}\sum_{i=1}^{N}\|\phi(\mub_{i})-\eta \Psi'(Pu_{\mub_{i}})\|^{2}}+\lambda\regularizer(\phi).
\end{equation}
Let $f_{\eta}:\paramdomain\to\mathbb{R}^{\tilde{m}}$ be defined as $f_{\eta}(\mub):=\eta Tu_{\mub}=\eta T\operator(\mub).$ Let $\polyellipse_{\brho}$ be the Bernstein polyellipse in Definition~\ref{def:hidden} corresponding to $\operator\in\ha_{\gamma,\epsilon,s}(\paramdomain).$ Since $T$ is linear, and thus entire, it is clear that $f_{\eta}$ admits a holomorphic extension to $\polyellipse_{\brho}$. Furthermore, by composition,
\[\|f_{\eta}\|_{L^{\infty}(\polyellipse_{\brho}, \mathbb{R}^{\tilde{m}})}\le\eta\opnorm{T}\cdot\|\operator\|_{L^{\infty}(\polyellipse_{\brho}, H^{s}(\Omega))}\le 2\eta\|\operator\|_{L^{\infty}(\polyellipse_{\brho}, H^{s}(\Omega))}.\]
In light of this, hereon we shall fix the rescaling parameter to
\[\eta_{*}:=\frac{1}{4\|\operator\|_{L^{\infty}(\polyellipse_{\brho}, H^{s}(\Omega))}},\]
so that $f:=f_{\eta_{*}}$ satisfies $\|f\|_{L^{\infty}(\polyellipse_{\brho}, \mathbb{R}^{\tilde{m}})}\le1/2.$ We now recall that, thanks to \eqref{eq:deltaclose}, we also have 
\[\sup_{\mub\in\paramdomain}\|f(\mub)-\eta_{*}\Psi'(P u_{\mub})\|_{2}<\eta_{*}\delta.\]
This allows us to interpret $\eta_{*}\Psi'(Pu_{\mub_{i}})$ as
perturbations of $f(\mub_{i})$, and thus consider
Problem \eqref{eq:rescaled} as the training of a neural network model with ground truth $f$ and noisy samples $\eta_{*}\Psi'(Pu_{\mub_{i}})\approx f(\mub_{i})$. In particular, by applying \cite[Theorem 5]{adcock2022deep} to $f$ and \eqref{eq:rescaled} with $\eta=\eta_{*}$, we see that the loss minimizer $\hat{\phi}_{\eta_{*}}$ satisfies
\begin{equation}
    \label{eq:practicalbound}
    \mathbb{E}_{\mub\sim\probability}^{1/2}\|f(\mub)-\hat{\phi}_{\eta_{*}}(\mub)\|^{2}_{2}\le c_{4}\exp\left(-\frac{1}{\sqrt{2}}\gamma\tilde{N}^{1/(2p)}\right)+c_{4}\eta_{*}\delta,
    \end{equation}
with probability $1-\varepsilon$, for all $\tilde{N}\ge N_{0}$, where $N_{0}=N_{0}(\gamma,p,f)$ is a lower bound on the size of the training set. 
Let now $\phi$ be (any of) the original minimizer in the Theorem. As noted in \eqref{eq:rescaling}, we have $\hat{\phi}=\hat{\phi}_{\eta_{*}}\cdot\eta_{*}^{-1}$ for some minimizer $\phi_{\eta_*}$ of the rescaled problem \eqref{eq:rescaled}. Thus, 
\begin{multline}
    \label{eq:practicalbound2}
    \mathbb{E}_{\mub\sim\probability}^{1/2}\|Tu_{\mub}-\hat{\phi}(\mub)\|^{2}_{2}=\\=\eta_{*}^{-1}\mathbb{E}_{\mub\sim\probability}^{1/2}\|\eta_{*}Tu_{\mub}-\hat{\phi}_{\eta_{*}}(\mub)\|^{2}_{2}\le\\\le \review{\eta_{*}^{-1}c_{4}\exp\left(-\frac{1}{\sqrt{2}}\gamma\tilde{N}^{1/(2p)}\right)+c_4\delta},
\end{multline}
as $f(\mub)=\eta_{*}Tu_{\mub}.$
\end{proof}

\noindent\textbf{Step 6.} \textit{The error bound in \eqref{eq:mainresult} holds true.}
\begin{proof}
Let $\|\cdot\|_{1}$ denote the $1$-norm on $\mathbb{R}^{\tilde{m}}$, so that $\|\mathbf{a}\|_{1}:=\sum_{i=1}^{\tilde{m}}|a_{j}|.$ We recall that the following hold
\begin{equation}
\label{eq:l1diseq}
\|\mathbf{a}-\mathbf{b}\|_{1}\le\sqrt{\tilde{m}}\|\mathbf{a}-\mathbf{b}\|_{2},\quad\quad|\Psi_{j}(\mathbf{a})-\Psi_{j}(\mathbf{b})|\le\|\mathbf{a}-\mathbf{b}\|_{1}.
\end{equation}
For the interested reader, we refer to \cite{franco2023approximation}, Pag. 7, for a detailed proof of the second inequality. 
We also note that, by definition,
\begin{equation}
    \label{eq:normbound}
    \|u_{\mub}\|_{H^{s}(\Omega)}\le \|\operator\|_{L^{\infty}(\polyellipse_{\brho}, H^{s}(\Omega))}.
\end{equation}
\review{To ease notation, let
$$E:=\expe_{\mub\sim\probability}^{1/2}\left[\sup_{j}|u_{\mub}(x_{j})-\Psi_{j}(\hat{\phi}(\mub))|^{2}\right]$$}
Since
\begin{equation*} \review{E}\le
\expe_{\mub\sim\probability}^{1/2}\left[\sup_{j}|u_{\mub}(x_{j})-\Psi_{j}(Tu_{\mub})|^{2}\right]+\expe_{\mub\sim\probability}^{1/2}\left[\sup_{j}|\Psi_{j}(Tu_{\mub})-\Psi_{j}(\hat{\phi}(\mub))|^{2}\right],
\end{equation*}
combining \eqref{eq:autoencoder}, \eqref{eq:practicalbound2}, \eqref{eq:l1diseq} and \eqref{eq:normbound}, ultimately yields
\begin{multline*}
\review{E}\le
\sqrt{\frac{2m^{1-2s}}{2s-1}}\|\operator\|_{L^{\infty}(\polyellipse_{\brho}, H^{s}(\Omega))}+\expe_{\mub\sim\probability}^{1/2}\|Tu_{\mub}-\review{\hat{\phi}}(\mub)\|^{2}_{1}\le\\
\le\sqrt{\frac{2m^{1-2s}}{2s-1}}\|\operator\|_{L^{\infty}(\polyellipse_{\brho}, H^{s}(\Omega))}+\sqrt{\tilde{m}}\expe_{\mub\sim\probability}^{1/2}\|Tu_{\mub}-\review{\hat{\phi}}(\mub)\|^{2}_{2},
\end{multline*}  
and thus,
\begin{equation}
\label{eq:conclusion0}
\review{E}\le\left(
\sqrt{\frac{2m^{1-2s}}{2s-1}} + \review{4\sqrt{\tilde{m}}
c_{4}\exp\left(-\frac{1}{\sqrt{2}}\gamma\tilde{N}^{1/(2p)}\right)+c_4\sqrt{\tilde{m}}g^{-1}\delta}\right)g
\end{equation} 
where, for better readability, we have \review{set $g:=\|\operator\|_{L^{\infty}(\polyellipse_{\brho}, H^{s}(\Omega))}$. Let us now fix the value of $\delta>0$ such that
\begin{equation}
    \delta \le \frac{g}{c_4}\sqrt{\frac{2m^{1-2s}}{(2s-1)\tilde{m}}}.
\end{equation}
Then, \eqref{eq:conclusion0} can be simplified to}
\begin{equation*}
\review{E\le\left(
2\sqrt{\frac{2m^{1-2s}}{2s-1}} + \review{4\sqrt{\tilde{m}}
c_{4}\exp\left(-\frac{1}{\sqrt{2}}\gamma\tilde{N}^{1/(2p)}\right)}\right)g.}
\end{equation*}
\noindent Since $\sqrt{\tilde{m}}\le\sqrt{4m}$, it \review{follows that,}
\begin{equation}
\label{eq:conclusion}
\review{E\le\max\{2,8c_4\}\left(
\sqrt{\frac{2m^{1-2s}}{2s-1}} + \sqrt{m}
\exp\left(-\frac{1}{\sqrt{2}}\gamma\tilde{N}^{1/(2p)}\right)\right)g.}
\end{equation}
\review{In  particular, up to re-naming the universal constant as $c_{4}':=\max\{2,8c_4\}$}, Eq.~\eqref{eq:conclusion} immediately yields the desired conclusion.
\end{proof}


\section{Conclusion}
\label{sec:conclusions}

Motivated by the empirical success of deep-learning-based reduced order models for parametric PDEs, we  proposed a new  \emph{practical existence theorem} (Theorem~\ref{theorem:cnn}). \review{Our analysis focuses on models relying on deep autoencoders, where two networks, $\Psi'$ and $\Psi$ are used to compress the output, whereas a third network, $\phi$, is used to learn the parameter-to-latent-variables map; the parameter-to-solution operator $\operator$ is then approximated via composition, $\operator\approx\Psi\circ\phi$.} Focusing on the case of deep convolutional autoencoders, our theorem provides an explicit error bound for trained models 
\review{in which the decoder $\Psi$ is constructed explicitly and the reduced network $\phi$ is trained via regularized empirical loss minimization. In doing so, the theorem also provides a list of sufficient conditions on the overall complexity of the reduced order model, $\Psi\circ\phi$, as well as detailed information concerning the training phase of the reduced network $\phi$ (sample size, sampling strategy, choice of the loss function). Notably,} 
our theorem validates several  heuristic observations from previous numerical studies, hence reducing the gap between theory and practice in this fast-growing area.

We conclude by mentioning some limitations of our theory, whose study is left to future work.  First, our theory only covers the case of one-dimensional physical domains. Generalizing the theory to higher dimensions is an important open question. \review{In this regard, there are two main obstacles that hinder the extensibility of our analysis to $d>1$. The first one is a technicality regarding the operator $T$ in the proof of Theorem \ref{theorem:cnn} (see also Lemma \ref{lemma:T}). Simply put, the latter consists of a periodicization operator, mapping arbitrary functions onto smooth periodic signals, composed with the truncated Fourier transform for data compression. Adapting this idea for $d>1$ is nontrivial since domains can have arbitrary shapes. One idea could be to embed $\Omega$ onto a suitable hypercube $[-1,1]^{d}$ and, depending on the smoothness of $\partial\Omega$, leverage well-known estimates of Sobolev extension operators, see, e.g., \cite{fefferman2014sobolev}. The second issue, instead, is merely practical. Our construction relies on a convolutional architecture that can replicate the performances of the Fourier transform. In higher-dimensions, this requires a very careful adaptation of \cite[Lemma 1-4]{franco2023approximation}.}



\review{Aside from this, another}
interesting line of future work is the study of further parametric PDEs with holomorphic parametric dependence, such as elliptic problems with higher regularity ($s>1$), parabolic problems,  or PDEs over parametrized domains (see, e.g., \cite{cohen2018shape}). 
\review{Here, in fact, we only discussed the application of our theory to the parametric diffusion equation (see Section~\ref{sec:param_diff}).}

\review{As we mentioned previously, it is worth remarking  that} Theorem~\ref{theorem:cnn} only addresses the training of the reduced network $\phi$. 
The 
encoder $\Psi$ and the decoder $\Psi'$, instead, are constructed \review{either using universal approximation theorems (encoder), or explicitly (decoder), that is, by hard-coding all weights and biases in the architecture}. 
\review{In addition, the fact that $\phi$ has standard as opposed to convolutional is inherited from \cite[Theorem~5]{adcock2022deep}, upon which Theorem~\ref{theorem:cnn} relies.}

\review{We conclude with some comments on the curse of dimensionality. First, we observe that the sample complexity is affected by the curse in Theorem~\ref{theorem:cnn}. In fact, considering the term involving $\tilde{N}$ in \eqref{eq:mainresult}, for a given target accuracy $\tau>0$, one has $\sqrt{m}\exp(-\gamma \tilde{N}^{1/(2p)}/\sqrt{2}) \leq \tau$ if anly only if $\tilde{N} \geq (\sqrt{2}\log(\sqrt{m}/\tau)/\gamma)^{2p}$, which leads to an exponential dependence of $\tilde{N}$ on $p$. Moreover, our results lose significance if $p \approx N_h $, mainly due to the curse of dimensionality affecting the reduced network $\phi:\mathbb{R}^{p} \to \mathbb{R}^{\tilde{m}}$. In fact, our complexity bounds include the exponential term $\tilde{N}^{1/(2p)}$. In principle, this issue could be addressed in---at least---three ways. The first one could be to focus on a smaller class of operators, that is, analytic parameter-to-solution maps enjoying suitable summability properties in their power series expansion, as in \cite{lanthaler2022error, schwab2023deep}. Second, one could consider proving a practical existence theorem using algebraic as opposed to exponential best $s$-term decay rates (see \cite[Chapter~3]{adcock2022sparse}) in the spirit of \cite[Theorem 8.1]{adcock2024learning}. This would also have to be combined (similarly to the first strategy) with higher regularity assumptions involving infinite-dimensional analyticity and would require an adaptation of the argument in \cite{adcock2024learning} to the Hilbert- or, at least, vector-valued setting.  A third approach could rely on incorporating an additional compression phase at input, either through autoencoders or linear projections; see, e.g., \cite{franco2023latent, herrmann2022neural, lanthaler2023operator}. Nevertheless, adapting these ideas to our context is challenging due to the need for a theory describing the implementation and training of a (convolutional) encoder; as we mentioned previously, this is, in general, highly nontrivial.}






\section*{Acknowledgments}
SB acknowledges the support of the Natural Sciences and Engineering Research Council
of Canada (NSERC) through grant RGPIN-2020-06766 and the Fonds de Recherche du Qu\'ebec Nature et Technologies (FRQNT) through grant 313276.
NF is member of the Gruppo Nazionale per il Calcolo Scientifico (GNCS) of the Istituto Nazionale di Alta Matematica (INdAM). The present research is part of the activities of project Dipartimento di Eccellenza 2023-2027, Department of Mathematics, Politecnico di Milano, funded by MUR, and of project Cal.Hub.Ria (Piano Operativo Salute, traiettoria 4), funded by MSAL.

The authors would also like to thank Prof.\ Paolo Zunino (Politecnico di Milano) for promoting the development and publication of this work \review{and Prof.\ Ben Adcock (Simon Fraser University) for providing helpful comments on a earlier version of this manuscript}.

\appendix
\section{Hermite polynomials and signal periodicization}
\label{appendix:hermite}
\newcommand{\psjbound}{(\sqrt{1/2})^{j+1}}

\newtheorem{lemmaA}{Lemma}[section]

This Appendix presents two supplementary results, both of which are essential for our construction. In particular, we expand on the definition of the operator $T$ appearing in the proof of Theorem \ref{theorem:cnn}, while simultaneously deriving some useful inequalities. 

Following \cite{franco2023approximation}, our approach involves employing a periodicization operator that leverages on Hermite interpolation: see Fig.~\ref{fig:periodicization} for a visual representation. Thus, we first derive some preliminary results related to Hermite polynomials (Lemma~\ref{lemma:hermite}), and then proceed with a synthetic discussion about the definition and the analytical properties of the operator $T$ (Lemma~\ref{lemma:T}).

\begin{lemmaA}
    \label{lemma:hermite}
    Let $s\in\mathbb{N}$, $s\ge1$. For any $0\le j\le s-1$, let $p_{s,j}$ and $q_{s,j}$ be the unique polynomials of degree $2s-1$ for which the following hold true
    \[\begin{array}{lll}
         p_{s,j}^{(k)}(0)=\delta_{j,k}& &p_{s,j}^{(k)}(1)=0 \\\\
         q_{s,j}^{(k)}(0)=0& &q_{s,j}^{(k)}(1)=\delta_{j,k},
    \end{array}\]
    Then, $\|q_{s,j}\|_{L^{2}(0,1)}=\|p_{s,j}\|_{L^{2}(0,1)}$. Furthermore, $\|p_{s,j}\|_{L^{2}(0,1)}\le \psjbound.$
\end{lemmaA}

\begin{proof}
    Since $q_{s,j}(x)=(-1)^{j}p_{s,j}(1-x)$, the first statement is obvious. As for the second one, we shall proceed in three steps.\\\\
    \textbf{Step 1.} \emph{We prove that $p_{s,j}(x)\ge0\;\forall x\in[0,1].$}\vspace{0.25cm}\\
    We note that, since $p_{s,j}$ has a zero of order $s$ at 1, we have
    \[p_{s,j}(x)=g(x)(1-x)^{s}\]
    for some polynomial $g$ of degree $s-1$, which depends on $s$ and $j$. We now notice that, since $g(x)=p_{s,j}(x)(1-x)^{-s}$, one has 
    \[g^{(k)}(x)=\sum_{l=0}^{k}\binom{k}{l}p_{s,j}^{(k-l)}(x)(1-x)^{-s-l}\frac{(s+l)!}{s!}.\]
    Let $a_{k}$ be the $k$th coefficient in the polynomial expansion of $g$. Then, the above implies
    \[a_{k}=\frac{1}{k!}g^{(k)}(0)=\frac{1}{k!}\sum_{l=0}^{k}\binom{k}{l}p_{s,j}^{(k-l)}(0)\frac{(s+l)!}{s!}\ge0,\]
    since $p_{s,j}^{(l)}(0)\ge0$ for all $0\le l\le k\le s-1$. In particular, all the coefficients in $g$ are positive, implying $g\ge0$ on $[0,+\infty)$, and thus $p_{s,j}\ge0$ on $[0,1]$, as claimed.
    \\\\
    \textbf{Step 2.} \emph{We prove that $\|p_{s,0}\|_{L^{2}(0,1)}\le\sqrt{1/2}.$}\vspace{0.25cm}\\
    Using the definition, it is straightforward to verify that the polynomial $p_{s,0}$ can be written in closed form as
    \begin{equation}
        \label{eq:ps0integral}
        p_{s,0}(x)=1-\frac{\int_{0}^{x}y^{s-1}(1-y)^{s-1}dy}{\int_{0}^{1}y^{s-1}(1-y)^{s-1}dy}.
    \end{equation}    
    In fact, the right-hand-side of \eqref{eq:ps0integral}: i) is a polynomial of degree $(s-1)+(s-1)+1=2s-1$; ii) vanishes at $x=1$, while it equals 1 at $x=0$; iii) its derivative is proportional to $x^{s-1}(1-x)^{s-1}$, which vanishes at $x=0,1$ with all its higher order derivatives (up to degree $s-2$). Since the polynomial $p_{s,0}$ is uniquely characterized by such conditions, this proves that the identity in \eqref{eq:ps0integral} holds true.
    
    We now note that, since the integrand $y\mapsto y^{s-1}(1-y)^{s-1}$ is positive, the polynomial $p_{s,0}$ happens to be monotone nonincreasing in $[0,1]$. Consequently, 
    \[0\le p_{s,0}(x)\le p_{s,0}(0)= 1,\]
    for all $x\in[0,1]$, and thus
    \begin{equation}
        \label{eq:l1l2}
            \|p_{s,0}\|_{L^{2}(0,1)}^{2} = \int_{0}^{1}p_{s,0}^{2}(x)dx\le\int_{0}^{1}p_{s,0}(x)dx
    \end{equation}
    Furthermore, due symmetry, it is straightforward to see that
    \[p_{s,0}(x)=1-p_{s,0}(1-x),\]
    from which, up to a simple change of variables, it follows that
    \[\int_{0}^{1}p_{s,0}(x)dx=1-\int_{0}^{1}p_{s,0}(1-x)dx=1+\int_{1}^{0}p_{s,0}(z)dz=1-\int_{0}^{1}p_{s,0}(z)dz\]
    \[\implies \int_{0}^{1}p_{s,0}(x)dx=\frac{1}{2},\]
    which in turn implies $\|p_{s,0}\|_{L^{2}(0,1)}\le\sqrt{1/2}$ due to \eqref{eq:l1l2}.
    \\\\
    \textbf{Step 3.} \emph{We prove that $\|p_{s,j}\|_{L^{2}(0,1)}\le\psjbound.$}\vspace{0.25cm}\\
    To prove the remaining cases, we shall exploit the following recursive formula,  
    \[    
    p_{s,j}(x)=\int_{0}^{x}p_{s-1,j-1}(y)dy+\left(p_{s,0}(x)-1\right)\int_{0}^{1}p_{s-1,j-1}(y)dy,
    \]
    which can be easily verified by hand. We re-write the above as
      \begin{equation}
      p_{s,j}(x)=p_{s,0}(x)\int_{0}^{1}p_{s-1,j-1}(y)dy-\int_{x}^{1}p_{s-1,j-1}(y)dy.   \end{equation}
   Since all polynomials in the form $p_{\tilde{s},\tilde{j}}$ are positive (cf. Step 1), we have
   \[0\le p_{s,j}(x)\le p_{s,0}(x)\int_{0}^{1}p_{s-1,j-1}(y)dy,\]
   implying that,
   \[\|p_{s,j}\|_{L^{2}(0,1)}\le \|p_{s,0}\|_{L^{2}(0,1)}\int_{0}^{1}p_{s-1,j-1}(y)dy\le \|p_{s,0}\|_{L^{2}(0,1)}\|p_{s-1,j-1}\|_{L^{2}(0,1)}.\]
   Finally, iterating the above and applying the result at Step 2, yields
   \[\|p_{s,j}\|_{L^{2}(0,1)}\le \|p_{s,0}\|_{L^{2}(0,1)}\cdot \|p_{s-1,0}\|_{L^{2}(0,1)}\cdot\ldots\cdot \|p_{s-j,0}\|_{L^{2}(0,1)}\le\psjbound.\]
\end{proof}

\begin{lemmaA}
    \label{lemma:T}
    Let  $\Omega:=(0,1).$ Let $s,m\in\mathbb{N}$, $s,m\ge1$. For any $f\in H^{s}(\Omega)$, let $p_{f}$ be the polynomial of degree $2s-1$ given by
    \[p_{f}(x):=\sum_{j=0}^{s-1}[f^{(j)}(1)-f^{(j)}(0)]\cdot\left[p_{s,j}(x)-q_{s,j}(x)\right],\]
    and let $\tilde{f}\in H^{s}(\Omega)$ be the periodicized version of $f$, which we define as (cf. Fig.~\ref{fig:periodicization})
    \begin{equation}
        \label{eq:Tdef}
        \tilde{f}(x):=
        \begin{cases}
        f(2x)+p_{f}(2x) & 0\le x\le 1/2\\
        f(2x-1) & 1/2<x\le1.
        \end{cases}
    \end{equation}
    Define the linear operator $T:H^{s}(\Omega)\to\mathbb{C}^{2m+1}$ as
    \[T: f\mapsto \left[\int_{0}^{1}\tilde{f}(x)e^{\review{2\pi\mathbf{i}mx}}dx,\dots,\int_{0}^{1}\tilde{f}(x)e^{-2\pi\mathbf{i}mx}dx\right].\]
    Then, $\opnorm{T}\le 2.$
\end{lemmaA}

\begin{proof}
Let $f\in H^{s}(\Omega)$. For any $j=0,\dots,s-1$ we have
\[|f^{(j)}(1)-f^{(j)}(0)|=\left|\int_{0}^{1}f^{(j+1)}(x)dx\right|\le\|f^{(j+1)}\|_{L^{2}(\Omega)}.\]
By Lemma~\ref{lemma:hermite}, we have
\[\|p_{f}\|_{L^{2}(\Omega)}\le 2\sum_{j=0}^{s-1}|f^{(j)}(1)-f^{(j)}(0)|\cdot\|p_{s,j}\|_{L^{2}(\Omega)}
\le2\sum_{j=0}^{s-1}\psjbound\|f^{(j+1)}\|_{L^{2}(\Omega)}.\]
Then, by the Cauchy-\review{Schwarz} inequality,
\begin{align*}
    \|p_{f}\|_{L^{2}(\Omega)}&\le 2\sqrt{\sum_{j=0}^{s-1}\left(\frac{1}{2}\right)^{j+1}}\sqrt{\sum_{j=0}^{s-1}\|f^{(j+1)}\|_{L^{2}(\Omega)}^{2}}\le \\ &\le 2\sqrt{\sum_{j=0}^{+\infty}\left(\frac{1}{2}\right)^{j+1}}\|f\|_{H^{s}(\Omega)}= 2\|f\|_{H^{s}(\Omega)}.
    \end{align*}
Consequently,
\[
    \|f+p_{f}\|_{L^{2}(\Omega)}\le\|f\|_{L^{2}(\Omega)}+\|p_{f}\|_{L^{2}(\Omega)}\le 3\|f\|_{H^{s}(\Omega)}.
\]
We now note that a simple change of variables yields
\[\|\tilde{f}\|_{L^{2}(\Omega)}^{2}=\frac{1}{2}\|f+p_{f}\|_{L^{2}(\Omega)}^{2}+\frac{1}{2}\|f\|_{L^{2}(\Omega)}^{2},\]
implying that
\[\|\tilde{f}\|_{L^{2}(\Omega)}\le \frac{3}{2}\|f\|_{H^{s}(\Omega)}+\frac{1}{2}\|f\|_{L^{2}(\Omega)}\le 2\|f\|_{H^{s}(\Omega)}.\]
Finally, we note that $T$ maps $f$ onto the truncated Fourier coefficients of $\tilde{f}$. In particular, for $\|\cdot\|_{2}$ the Euclidean norm,
\[\|Tf\|_{2}\le\|\tilde{f}\|_{L^{2}(\Omega)}\le2\|f\|_{H^{s}(\Omega)},\]
as claimed.
\end{proof}

\begin{figure}
    \centering
    \includegraphics[width=\textwidth]{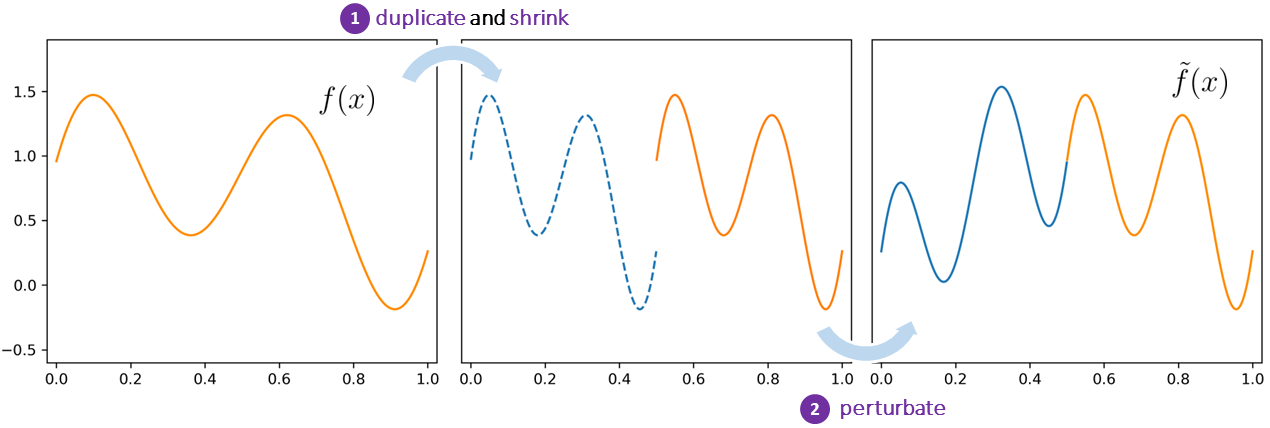}
    \caption{Visualization of the transformation $f\mapsto \tilde{f}$ used in Lemma~\ref{lemma:T}. The signal $f$ is duplicated and a polynomial perturbation is added to ensure (smooth) periodicity.}
    \label{fig:periodicization}
\end{figure}

\section{Auxiliary results on ReLU networks}
\label{appendix:relu}

This Appendix contains some technical details about the interplay between linear networks and ReLU networks. As noted in the proof of Theorem~\ref{theorem:cnn}, these considerations are fundamental, as they allow us to adapt \cite[Theorem 1]{franco2023approximation} to our setting. 

\newcommand{\cb}{\mathbf{c}}
\begin{lemmaA}
    \label{lemma:relu} Let $\Psi:\mathbb{R}^{m}\to\mathbb{R}^{N}$ be a linear network (no activations nor biases at any level). For every compact set $C\subset\mathbb{R}^{m}$, there exists a ReLU network $\tilde{\Psi}$ having the same architecture (and the same weights), such that $\tilde{\Psi}(\cb)=\Psi(\cb)$ for all $\cb\in C.$
\end{lemmaA}

\begin{proof}
    In plain words, the idea is to introduce suitable biases at the internal layers that can shift neuron entries to nonnegative values (which would be unaffected by ReLUs). Then, a terminal bias is used to shift the output back to the desired value. We shall now discuss the whole idea in a more rigorous way. 
    Let $\ell$ be the number of hidden layers in $\Psi$. Since $\Psi$ is linear, it must be of the form
    \[\Psi(\cb)=\weight_{\ell+1}\cdot \dots\cdot \weight_{1}\cb,\]
    where $\weight_{i}$, $i=1,\dots,l+1$, are the matrices representing the action of the $i$th layer, respectively. 
    Let us introduce the following notation
    \[\weight_{i\to j}:=\prod_{k=i}^{j}\weight_{k},\]
    defined for all pairs $1\le i\le j\le \ell+1.$
    We construct a sequence of biases $\mathbf{b}_{0},\dots,\mathbf{b}_{\ell+1},$ via the iterative scheme below,
    \begin{equation}
    \label{eq:biases}
        \begin{cases}
            \mathbf{b}_{0}=0,&\\
            \mathbf{b}_{i}=\displaystyle-\min_{\cb\in C}\left(\weight_{1\to i}\cb+\sum_{k=0}^{i-1}\weight_{k+1\to i}\mathbf{b}_{k}\right),&i=1,\dots,\ell,\\
            \mathbf{b}_{\ell+1}=\displaystyle -\sum_{k=0}^{\ell}\weight_{k+1\to \ell+1}\mathbf{b}_{k}&
        \end{cases}
    \end{equation}
    the minimum being defined entrywise (note that all minima are well-defined due compactness of $C$).
    For $\dnnactivation$ the ReLU activation function, consider the layers
    \[L_{i}: x\mapsto \dnnactivation\left(\weight_{i}\cb+\mathbf{b}_{i}\right),\]
    defined for $i=1,\dots,\ell.$ We claim that the ReLU network
    \[\tilde{\Psi}:=\weight_{\ell+1}(L_{\ell}\circ\dots\circ L_{1})(\cb)+\mathbf{b}_{\ell+1}\]
    coincides with $\Psi$ over $C.$
    To see this, we start by noting that for all $\cb\in C$, due to \eqref{eq:biases}, we have
    \[\mathbf{b}_{1}\ge - \weight_{1\to 1}\cb -\cancel{\weight_{1\to1}\mathbf{b}_{0}}=-\weight_{1}\cb,\]
    implying that $\weight_{1}\cb+\mathbf{b}_{1}$ has nonnegative entries. Consequently,
    \[L_{1}(\cb)=\dnnactivation\left(\weight_{1}\cb+\mathbf{b}_{1}\right)=\weight_{1}\cb+\mathbf{b}_{1}.\]
    Similarly,
    \[\mathbf{b}_{2}\ge - \weight_{1\to 2}\cb -\cancel{\weight_{1\to2}\mathbf{b}_{0}}-\weight_{2\to2}\mathbf{b}_{1}=-\weight_{2}\weight_{1}\cb-\weight_{2}\mathbf{b}_{1},\]
    implying that
    \[\weight_{2}\weight_{1}\cb+\weight_{2}\mathbf{b}_{1}+\mathbf{b}_{2},\]
    has nonnegative entries, and thus
    \[L_{2}(L_{1}(\cb))=L_{2}(\weight_{1}\cb+\mathbf{b}_{1})=\dnnactivation\left(\weight_{2}\weight_{1}\cb+\weight_{2}\mathbf{b}_{1}+\mathbf{b}_{2}\right)=\weight_{2}\weight_{1}\cb+\weight_{2}\mathbf{b}_{1}+\mathbf{b}_{2}.\]
    Iterating the above argument, one can easily see that
    \begin{multline*}
    \left(L_{\ell}\circ\dots\circ L_{1}\right)(\cb)=\weight_{\ell}\cdot\dots\cdot\weight_{1}\cb + \weight_{\ell}\cdot\dots\cdot\weight_{2}\mathbf{b}_{1}+\weight_{\ell}\cdot\dots\cdot\weight_{3}\mathbf{b}_{2}+\dots+\mathbf{b}_{\ell}=\\=\weight_{1\to \ell}\cb + \sum_{k=0}^{\ell-1}\weight_{k+1\to \ell}\mathbf{b}_{k}+\mathbf{b}_{\ell},\end{multline*}
    for all $\cb\in C$. Then,
    \begin{multline*}
    \tilde{\Psi}(\cb)=\weight_{\ell+1}\weight_{1\to \ell}\cb + \sum_{k=0}^{\ell-1}\weight_{\ell+1}\weight_{k+1\to \ell}\mathbf{b}_{k}+\weight_{\ell+1}\mathbf{b}_{\ell}+\mathbf{b}_{\ell+1}=\\=
    \weight_{1\to \ell+1}\cb + \sum_{k=0}^{\ell-1}\weight_{k+1\to \ell+1}\mathbf{b}_{k}+\weight_{\ell+1}\mathbf{b}_{\ell}+\mathbf{b}_{\ell+1}=\\=
    \weight_{1\to \ell+1}\cb + \cancel{\sum_{k=0}^{\ell}\weight_{k+1\to \ell+1}\mathbf{b}_{k}}+\cancel{\mathbf{b}_{\ell+1}}=\weight_{1\to \ell+1}\cb=\Psi(\cb).\end{multline*}
    In particular, $\tilde{\Psi}_{|C}\equiv\Psi_{|C}$, as wished.
\end{proof}

\newtheorem{corollaryA}{Corollary}[section]
\begin{corollaryA}
    \label{corollary:relu} Let $\Psi:\mathbb{R}^{m}\to\mathbb{R}^{N}$ be a linear CNN (no activations nor biases at any level). For every compact set $C\subset\mathbb{R}^{m}$, there exists a ReLU CNN $\tilde{\Psi}$ having the same architecture (and the same weights), such that $\tilde{\Psi}(\cb)=\Psi(\cb)$ for all $\cb\in C.$
\end{corollaryA}
\begin{proof}
    This is a direct consequence of Lemma~\ref{lemma:relu}. In fact, convolutional layers are uniquely characterized by the fact of having a linear component that acts as a convolution operator. Since, in the lemma, the transformation $\Psi\to\tilde{\Psi}$  preserves the linear part of each layer, the conclusion follows.
\end{proof}










\begin{thebibliography}{99}

\bibitem{adcock2022deep}
\newblock B. Adcock, S. Brugiapaglia, N. Dexter, and S. Moraga,
\newblock \textnormal{Deep neural networks are effective at learning high-dimensional Hilbert-valued functions from limited data},
\newblock In \emph{Proceedings of the 2nd Mathematical and Scientific Machine Learning Conference, Proceedings of Machine Learning Research}, \textbf{145} (2022), pp.~1-36.

\bibitem{adcock2022learning}
\newblock B. Adcock, S. Brugiapaglia, N. Dexter, and S. Moraga,
\newblock \textnormal{Near-optimal learning of Banach-valued, high-dimensional functions via deep neural networks},
\newblock \emph{arXiv preprint} arXiv:2211.12633 (2022).

\bibitem{adcock2024learning}
\newblock B. Adcock, S. Brugiapaglia, N. Dexter, and S. Moraga,
\newblock \textnormal{Learning smooth functions in high dimensions: From sparse polynomials to deep neural networks},
\newblock In \emph{Handbook of Numerical Analysis}, In press, (2024).

\bibitem{adcock2022sparse}
\newblock B. Adcock, S. Brugiapaglia, and C.G. Webster,
\newblock \textnormal{Sparse Polynomial Approximation of High-Dimensional Functions}, 
\newblock \emph{Society for Industrial and Applied Mathematics}, Philadelphia, PA, 2022

\bibitem{adcock2021gap}
\newblock B. Adcock and N. Dexter,
\newblock \textnormal{The gap between theory and practice in function approximation with deep neural networks},
\newblock \emph{SIAM Journal on Mathematics of Data Science}, \textbf{3}(2) (2021), pp.~624-655.

\bibitem{ahmed2020reduced}
\newblock S.E. Ahmed, S. Pawar, O. San, and A. Rasheed,
\newblock \textnormal{Reduced order modeling of fluid flows: Machine learning, Kolmogorov barrier, closure modeling, and partitioning},
\newblock In \emph{AIAA Aviation 2020 Forum} (2020), p.~2946.

\bibitem{barnett2022quadratic}
\newblock J. Barnett and C. Farhat,
\newblock \textnormal{Quadratic approximation manifold for mitigating the Kolmogorov barrier in nonlinear projection-based model order reduction},
\newblock \emph{Journal of Computational Physics}, \textbf{464} (2022), p.~111348.

\bibitem{boon2023deep}
\newblock W.M. Boon, N.R. Franco, A. Fumagalli, and P. Zunino,
\newblock \textnormal{Deep learning based reduced order modeling of Darcy flow systems with local mass conservation},
\newblock \emph{arXiv preprint} arXiv:2311.14554 (2023).

\bibitem{brivio2023error}
\newblock S. Brivio, S. Fresca, N.R. Franco, and A. Manzoni,
\newblock \textnormal{Error estimates for POD-DL-ROMs: a deep learning framework for reduced order modeling of nonlinear parametrized PDEs enhanced by proper orthogonal decomposition},
\newblock \emph{arXiv preprint} arXiv:2305.04680 (2023).

\bibitem{cheng2018deep}
\newblock Z. Cheng, H. Sun, M. Takeuchi, and J. Katto,
\newblock \textnormal{Deep convolutional autoencoder-based lossy image compression},
\newblock In \emph{2018 Picture Coding Symposium (PCS)}, June 2018, pp.~253-257. IEEE.

\bibitem{cohen2015approximation}
\newblock A. Cohen and R. DeVore,
\newblock \textnormal{Approximation of high-dimensional parametric PDEs},
\newblock \emph{Acta Numerica}, \textbf{24} (2015), pp.~1-159.

\bibitem{cohen2018shape}
\newblock A. Cohen, C. Schwab, and J. Zech,
\newblock \textnormal{Shape holomorphy of the stationary Navier--Stokes equations},
\newblock \emph{SIAM Journal on Mathematical Analysis}, \textbf{50}(2) (2018), pp.~1720-1752.

\bibitem{daubechies2022nonlinear}
\newblock I. Daubechies, R. DeVore, S. Foucart, B. Hanin, and G. Petrova,
\newblock \textnormal{Nonlinear approximation and (deep) ReLU networks},
\newblock \emph{Constructive Approximation}, \textbf{55}(1) (2022), pp.~127-172.

\bibitem{de1966splines}
\newblock C. De Boor and R.E. Lynch,
\newblock \textnormal{On splines and their minimum properties},
\newblock \emph{Journal of Mathematics and Mechanics}, \textbf{15}(6) (1966), pp.~953-969.


\bibitem{dugundji1951extension}
\newblock J. Dugundji,
\newblock \textnormal{An extension of Tietze's theorem},
\newblock \emph{Pacific J. Math.}, \textbf{1} (1951), pp.~353–367.

\bibitem{elbrachter2021deep}
\newblock D. Elbr{\"a}chter, D. Perekrestenko, P. Grohs, and H. B{\"o}lcskei,
\newblock \textnormal{Deep neural network approximation theory},
\newblock \emph{IEEE Transactions on Information Theory}, \textbf{67}(5) (2021), pp.~2581-2623.

\bibitem{evans2022partial}
\newblock L.C. Evans, 
\newblock \textnormal{Partial differential equations},
\newblock vol.~19 of Grad.\ Stud.\ Math., \emph{American Mathematical Society}, Providence, RI, 2nd ed., 2010.

\bibitem{fefferman2014sobolev}
\newblock C. Fefferman, A. Israel, and G. Luli,
\newblock \textnormal{Sobolev extension by linear operators},
\newblock \emph{Journal of the American Mathematical Society}, \textbf{27}(1) (2014), pp.~69--145.


\bibitem{franco2023latent}
\newblock N.R. Franco, D. Fraulin, A. Manzoni, and P. Zunino,
\newblock \textnormal{On the latent dimension of deep autoencoders for reduced order modeling of PDEs parametrized by random fields},
\newblock \emph{arXiv preprint} arXiv:2310.12095 (2023).

\bibitem{franco2023approximation}
\newblock N.R. Franco, S. Fresca, A. Manzoni and P. Zunino,
\newblock \textnormal{Approximation bounds for convolutional neural networks in operator learning},
\newblock \emph{Neural Networks}, \textbf{161} (2023), pp.~129-141.

\bibitem{franco2023deep}
\newblock N.R. Franco, A. Manzoni, and P. Zunino,
\newblock \textnormal{A deep learning approach to reduced order modelling of parameter dependent partial differential equations},
\newblock \emph{Mathematics of Computation}, \textbf{92}(340) (2023), pp.~483-524.

\bibitem{franco2023mesh-informed}
\newblock N.R. Franco, A. Manzoni, and P. Zunino,
\newblock \textnormal{Mesh-informed neural networks for operator learning in finite element spaces},
\newblock \emph{Journal of Scientific Computing}, \textbf{97}(2) (2023), 35.

\bibitem{frankle2019lottery}
\newblock J. Frankle and M. Carbin,
\newblock  The lottery ticket hypothesis: finding sparse, trainable neural networks,
\newblock In \emph{Proceedings of the 7th International Conference on Learning Representations (ICLR)} (2019).

\bibitem{fresca2022pod-dl-rom}
\newblock S. Fresca and A. Manzoni,
\newblock \textnormal{POD-DL-ROM: Enhancing deep learning-based reduced order models for nonlinear parametrized PDEs by proper orthogonal decomposition},
\newblock \emph{Computer Methods in Applied Mechanics and Engineering}, \textbf{388} (2022), p.~114181.

\bibitem{fresca2021comprehensive}
\newblock S. Fresca, L. Dede', and A. Manzoni,
\newblock \textnormal{A comprehensive deep learning-based approach to reduced order modeling of nonlinear time-dependent parametrized PDEs},
\newblock \emph{Journal of Scientific Computing}, \textbf{87} (2021), pp.~1-36.


\bibitem{hernandez2021deep}
\newblock Q. Hernandez, A. Badias, D. Gonzalez, F. Chinesta, and E. Cueto,
\newblock \textnormal{Deep learning of thermodynamics-aware reduced-order models from data},
\newblock \emph{Computer Methods in Applied Mechanics and Engineering}, \textbf{379} (2021), pp.~113763.

\bibitem{herrmann2022neural}
\newblock L. Herrmann, C. Schwab, and J. Zech,
\newblock \textnormal{Neural and GPC operator surrogates: construction and expression rate bounds},
\newblock \emph{arXiv preprint arXiv:2207.04950}, (2022).

\bibitem{hesthaven2018non-intrusive}
\newblock J.S. Hesthaven and S. Ubbiali,
\newblock \textnormal{Non-intrusive reduced order modeling of nonlinear problems using neural networks},
\newblock \emph{Journal of Computational Physics}, \textbf{363} (2018), pp.~55-78.

\bibitem{hornik1991approximation}
\newblock K. Hornik,
\newblock \textnormal{Approximation capabilities of multilayer feedforward networks},
\newblock \emph{Neural networks}, \textbf{4}(2) (1991), pp.~251-257.

\bibitem{kovachki2021universal}
\newblock N. Kovachki, S. Lanthaler, and S. Mishra,
\newblock \textnormal{On universal approximation and error bounds for Fourier neural operators},
\newblock \emph{The Journal of Machine Learning Research}, \textbf{22}(1) (2021), pp.~13237-13312.

\bibitem{kovachki2021neural}
\newblock N. Kovachki, Z. Li, B. Liu, K. Azizzadenesheli, K. Bhattacharya, A. Stuart, and A. Anandkumar,
\newblock \textnormal{Neural operator: Learning maps between function spaces},
\newblock \emph{arXiv preprint} arXiv:2108.08481 (2021).

\bibitem{kutyniok2022theoretical}
\newblock G. Kutyniok, P. Petersen, M. Raslan, and R. Schneider,
\newblock \textnormal{A theoretical analysis of deep neural networks and parametric PDEs},
\newblock \emph{Constructive Approximation}, \textbf{55}(1) (2022), pp.~73-125.

\bibitem{lanthaler2022error}
\newblock S. Lanthaler, S. Mishra, and G.E. Karniadakis,
\newblock \textnormal{Error estimates for DeePONets: A deep learning framework in infinite dimensions},
\newblock \emph{Transactions of Mathematics and Its Applications}, \textbf{6}(1), tnac001 (2022).

\bibitem{lanthaler2023operator}
\newblock S. Lanthaler,
\newblock \textnormal{Operator learning with PCA-Net: upper and lower complexity bounds},
\newblock \emph{Journal of Machine Learning Research}, \textbf{24}(318) (2023), pp.~1--67.

\bibitem{lecun1995convolutional}
\newblock Y. LeCun and Y. Bengio,
\newblock \textnormal{Convolutional networks for images, speech, and time series},
\newblock In \emph{The handbook of brain theory and neural networks}, MIT Press, \textbf{3361}(10) (1995).

\bibitem{lee2020model}
\newblock K. Lee and K.T. Carlberg,
\newblock \textnormal{Model reduction of dynamical systems on nonlinear manifolds using deep convolutional autoencoders},
\newblock \emph{Journal of Computational Physics}, \textbf{404} (2020), p.~108973.

\bibitem{li2020fourier}
\newblock Z. Li, N. Kovachki, K. Azizzadenesheli, B. Liu, K. Bhattacharya, A. Stuart, and A. Anandkumar,
\newblock \textnormal{Fourier neural operator for parametric partial differential equations},
\newblock \emph{arXiv preprint} arXiv:2010.08895 (2020).

\bibitem{liu2024generalization}
\newblock H. Liu, B. Dahal, R. Lai, W. Liao,
\newblock \textnormal{Generalization Error Guaranteed Auto-Encoder-Based Nonlinear Model Reduction for Operator Learning},
\newblock \emph{arXiv preprint arXiv:2401.10490},
2024.

\bibitem{lu2021learning}
\newblock L. Lu, P. Jin, G. Pang, Z. Zhang, and G.E. Karniadakis,
\newblock \textnormal{Learning nonlinear operators via DeepONet based on the universal approximation theorem of operators},
\newblock \emph{Nature Machine Intelligence}, \textbf{3}(3) (2021), pp.~218-229.

\bibitem{marcati2023exponential}
\newblock C. Marcati and C. Schwab,
\newblock \textnormal{Exponential convergence of deep operator networks for elliptic partial differential equations},
\newblock \emph{SIAM Journal on Numerical Analysis}, \textbf{61}(3) (2023), pp.~1513--1545.

\bibitem{masoumi-verki2023use}
\newblock S. Masoumi-Verki, F. Haghighat, N. Bouguila, and U. Eicker,
\newblock \textnormal{The use of GANs and transfer learning in model-order reduction of turbulent wake of an isolated high-rise building},
\newblock \emph{Building and Environment}, \textbf{246} (2023), p.~110948.

\bibitem{mishra2021enhancing}
\newblock S. Mishra and T. K. Rusch,
\newblock \textnormal{Enhancing Accuracy of Deep Learning Algorithms by Training with Low-Discrepancy Sequences},
\newblock \emph{SIAM Journal on Numerical Analysis}, \textbf{59}(3) (2021), pp.~1811-1834.

\bibitem{mucke2021reduced}
\newblock N.T. Mücke, S.M. Bohté, and C.W. Oosterlee,
\newblock \textnormal{Reduced order modeling for parameterized time-dependent PDEs using spatially and memory aware deep learning},
\newblock \emph{Journal of Computational Science}, \textbf{53} (2021), p.~101408.

\bibitem{oostwal2021hidden}
\newblock E. Oostwal, M. Straat, and M. Biehl,
\newblock \textnormal{Hidden Unit Specialization in Layered Neural Networks: ReLU vs. Sigmoidal Activation},
\newblock \emph{Physica A: Statistical Mechanics and its Applications}, \textbf{564} (2021), p.~125517.

\bibitem{peherstorfer2022breaking}
\newblock B. Peherstorfer,
\newblock \textnormal{Breaking the Kolmogorov barrier with nonlinear model reduction},
\newblock \emph{Notices of the American Mathematical Society}, \textbf{69}(5) (2022), pp.~725-733.

\bibitem{petersen2020equivalence}
\newblock P. Petersen and F. Voigtlaender,
\newblock \textnormal{Equivalence of approximation by convolutional neural networks and fully-connected networks},
\newblock \emph{Proceedings of the American Mathematical Society}, \textbf{148}(4) (2020), pp.~1567-1581.

\bibitem{pichi2023graph}
\newblock F. Pichi, B. Moya, and J.S. Hesthaven,
\newblock \textnormal{A graph convolutional autoencoder approach to model order reduction for parametrized PDEs},
\newblock \emph{arXiv preprint} arXiv:2305.08573 (2023).

\bibitem{pichi2023artificial}
\newblock F. Pichi, F. Ballarin, G. Rozza, and J.S. Hesthaven,
\newblock \textnormal{An artificial neural network approach to bifurcating phenomena in computational fluid dynamics},
\newblock \emph{Computers \& Fluids}, \textbf{254} (2023), p.~105813.

\bibitem{romor2022non-linear}
\newblock F. Romor, G. Stabile, and G. Rozza,
\newblock \textnormal{Non-linear manifold ROM with convolutional autoencoders and reduced over-collocation method},
\newblock \emph{arXiv preprint} arXiv:2203.00360 (2022).

\bibitem{rosafalco2021online}
\newblock L. Rosafalco, M. Torzoni, A. Manzoni, S. Mariani, and A. Corigliano,
\newblock \textnormal{Online structural health monitoring by model order reduction and deep learning algorithms},
\newblock \emph{Computers \& Structures}, \textbf{255} (2021), p.~106604.

\bibitem{schwab2023deep}
\newblock C. Schwab and J. Zech,
\newblock \textnormal{Deep learning in high dimension: neural network expression rates for analytic functions in $L^2(\mathbb{R}^d, \gamma_d)$},
\newblock \emph{SIAM/ASA Journal on Uncertainty Quantification}, \textbf{11}(1) (2023), pp.~199--234.


\bibitem{sun2018deep}
\newblock C. Sun, M. Ma, Z. Zhao, S. Tian, R. Yan, and X. Chen,
\newblock \textnormal{Deep transfer learning based on sparse autoencoder for remaining useful life prediction of tool in manufacturing},
\newblock \emph{IEEE Transactions on Industrial Informatics}, \textbf{15}(4) (2018), pp.~2416-2425.

\bibitem{vitullo2024nonlinear}
\newblock P. Vitullo, A. Colombo, N.R. Franco, A. Manzoni, and P. Zunino,
\newblock \textnormal{Nonlinear model order reduction for problems with microstructure using mesh informed neural networks},
\newblock \emph{Finite Elements in Analysis and Design}, \textbf{229} (2024), p.~104068.


\bibitem{yarotsky2017error}
\newblock D. Yarotsky,
\newblock \textnormal{Error bounds for approximations with deep ReLU networks},
\newblock \emph{Neural Networks}, \textbf{94} (2017), pp.~103-114.

\end{thebibliography}
\end{document}